\documentclass[12pt, a4paper]{amsart}
\usepackage[hmargin=30mm, vmargin=25mm, includefoot, twoside]{geometry}
\usepackage[bookmarksopen=true]{hyperref}
\usepackage[normalem]{ulem}
\usepackage[french,english]{babel}
\usepackage{amssymb,verbatim}
\usepackage{latexsym}
\usepackage{mathrsfs}
\usepackage{xspace}
\usepackage{enumerate, paralist}
\usepackage[usenames,dvipsnames]{color}
\usepackage{txfonts, pxfonts}
\newlength{\textlarg}

\newtheorem{prop}{Proposition}[section]
\newtheorem{thm}[prop]{Theorem}
\newtheorem*{thm*}{Theorem}

\newtheorem*{addendum*}{Addendum}
\newtheorem{cor}[prop]{Corollary}
\newtheorem{lem}[prop]{Lemma}
\newtheorem{thmintro}{Theorem}

\newtheorem{corintro}[thmintro]{Corollary}

\newtheorem*{convention*}{Convention}
\theoremstyle{definition}
\newtheorem*{defn*}{Definition}

\newtheorem{remark}[prop]{Remark}

\newtheorem*{scholium*}{Scholium}
\theoremstyle{remark}
\newtheorem{example}[prop]{Example}
\newtheorem*{example*}{Example}
\numberwithin{equation}{section}

\newcommand{\fixbug}{}
\newcommand{\ro}{\varrho}

\newcommand{\RR}{\mathbf{R}}

\newcommand{\ZZ}{\mathbf{Z}}

\newcommand{\la}{\langle}
\newcommand{\ra}{\rangle}
\newcommand{\inv}{^{-1}}

\newcommand{\se}{\subseteq}

\newcommand{\lra}{\longrightarrow}

\newcommand{\cathyp}{{\upshape CAT($-1$)}\xspace}

\newcommand{\tangle}[2]
{\angle_\mathrm{T}(#1,#2)}
\newcommand{\aangle}[3]
{\angle_{#1}(#2,#3)}
\newcommand{\cangle}[3]
{\overline{\angle}_{#1}(#2,#3)}

 \DeclareMathOperator{\Ker}{Ker}

\DeclareMathOperator{\Isom}{Is}

\DeclareMathOperator{\LF}{Rad_{\mathscr{L\!E\!}}}
\newcommand{\bd}{\partial} 

\def\Aut{\mathop{\mathrm{Aut}}\nolimits}

\def\min{\mathop{\mathrm{min}}\nolimits}

\begin{document}

\title{Amenable hyperbolic groups}
\author[P-E.~Caprace]{Pierre-Emmanuel Caprace{\large $^\spadesuit$}}
\address{Universit\'e catholique de Louvain, IRMP, Chemin du Cyclotron 2, 1348 Louvain-la-Neuve, Belgium}
\email{pe.caprace@uclouvain.be}
\thanks{{$\spadesuit$} F.R.S.-FNRS Research Associate, supported in part by FNRS grant F.4520.11 and the European Research Council}
\author[Y.~Cornulier]{Yves \lowercase{de} Cornulier{\large $^\varheartsuit$}}
\address{Laboratoire de Math\'ematiques\\
B\^atiment 425, Universit\'e Paris-Sud 11\\
91405 Orsay\\France}
\email{yves.cornulier@math.u-psud.fr}
\thanks{{$\varheartsuit$} Supported by ANR project {\it ``QuantiT"} JC08\textunderscore 318197.}

\author[N.~Monod]{Nicolas Monod{\large $^\vardiamondsuit$}}
\address{EPFL\\
SB -- {\sc MathGeom} -- EGG, Station 8\\
        CH--1015 Lausanne, 
        Switzerland}
\email{nicolas.monod@epfl.ch}
\thanks{{$\vardiamondsuit$} Supported in part by the Swiss National Science Foundation and the European Research Council}

\author[R.~Tessera]{Romain Tessera{\large $^\clubsuit$}}
\address{UMPA, ENS Lyon\\46 All\'ee d'Italie\\ 69364 Lyon Cedex\\ France}
\email{rtessera@umpa.ens-lyon.fr}
\thanks{{$\clubsuit$} Supported in part by ANR project GGAA}

\date{First draft: February 2012; revised: September 2013.}
%









\begin{abstract}
We give a complete characterization of the locally compact groups that are non-elementary Gromov-hyperbolic and amenable.  They coincide with the class of mapping tori of discrete or continuous one-parameter groups of compacting automorphisms. We moreover give a description of all Gromov-hyperbolic locally compact groups with a cocompact amenable subgroup: modulo a compact normal subgroup, these turn out to be either rank one simple Lie groups, or automorphism groups of semi-regular trees acting doubly transitively on the set of ends. As an application, we show that the class of hyperbolic locally compact groups with a  cusp-uniform non-uniform lattice, is very restricted.
\end{abstract}

\maketitle
\fixbug

\setcounter{tocdepth}{1}    
\tableofcontents

\section{Introduction}
\subsection{From negatively curved Lie groups to amenable hyperbolic groups}
John Milnor~\cite{milnor1976curvatures} initiated the study of left-invariant Riemannian metrics on general Lie groups and observed that a connected Lie group admitting a left-invariant \emph{negatively curved} Riemannian metric is necessarily soluble; he asked about a more precise characterization. This was answered by E.~Heintze~\cite{Heintze}: a connected Lie group has a negatively curved left-invariant Riemannian metric if and only if it can be written as a semidirect product $N\rtimes_{\alpha}\mathbf{R}$, where the group $N$ is a (nontrivial) nilpotent Lie group, which is contracted by the action of positive elements of $\mathbf{R}$, i.e.~$\lim_{t\to +\infty}\alpha(t)x =1$ for all $x\in N$. 

All these groups thus constitute examples  of locally compact groups that are both \emph{amenable} and (non-elementary) \emph{Gromov-hyperbolic}. The purpose of the present paper is to study this more general class of groups. 

It should be emphasized that, although most works devoted to Gromov-hyperbol\-icity focus on finitely generated discrete groups, Gromov's original concept was designed to encompass more general metric groups. We shall mostly focus here on compactly generated locally compact groups; this point of view is in fact very natural, as the full isometry group of a Gromov-hyperbolic metric space might very well be non-discrete. The definition reads as follows: A locally compact group $G$ is \textbf{Gromov-hyperbolic} (or, for short, \textbf{hyperbolic}) if it admits a compact generating set such that the associated word metric is Gromov-hyperbolic. In particular hyperbolicity is invariant under quasi-isometries. The definition might look unfamiliar to readers used to deal with locally compact spaces, since the Cayley graph associated with a compact generating set  is in general far from locally finite; moreover the natural action of the group on its Cayley graph need not be continuous. This matter of fact is however mitigated by the following characterization, proved in Corollary~\ref{cor:eqprop} below: a locally compact $G$ is Gromov-hyperbolic if and only if it admits a continuous proper cocompact isometric action on a Gromov hyperbolic proper geodesic metric space.

Gromov~\cite[\S 3.1,\S 8.2]{Gro} divides hyperbolic groups into three classes.

\begin{itemize}
\item The visual boundary  $\bd G$ is empty. Then $G$ is compact.
\item The visual boundary $\bd G$ consists of two points. This holds if and only if $G$ has an infinite cyclic closed cocompact subgroup. Actually, this can be improved as follows (Proposition~\ref{pro:debout}): $G$ has a unique maximal compact normal subgroup $W$ such that $G/W$ is isomorphic to a cocompact group of isometries of the real line, namely isomorphic to $\mathbf{Z}$, $\mathbf{Z}\rtimes\{\pm 1\}$, $\mathbf{R}$, or $\mathbf{R}\rtimes\{\pm 1\}$.
\item The visual boundary is uncountable.
\end{itemize}
Hyperbolic groups belonging to the first two classes are called \textbf{elementary}  and the above description provides for them a largely satisfactory classification; we shall focus on non-elementary hyperbolic groups. For example, a semisimple real Lie group is non-elementary hyperbolic if and only if it has real rank one. All Heintze groups mentioned above are non-elementary hyperbolic.

\medskip
In order to  state our first result, we introduce the following terminology. An automorphism $\alpha \in \Aut(H)$ of a locally compact group $H$ is called \textbf{compacting} if there is some compact subset $V \subseteq H$ such that for each $g \in H$, we have $\alpha^n(g) \in V$ for all sufficiently large $n>0$. In the  special case where  $\lim_{n \to \infty} \alpha^n(g) = 1$ for all $g \in G$, we say that   $\alpha$ is  \textbf{contracting}. 

The following result provides a first characterization of amenable hyperbolic groups, in the spirit of Heintze's characterization. 

\begin{thmintro}\label{thm:Amen:intro}
A locally compact group is amenable and non-elementary hyperbolic if and only if it can be written as a semidirect product $H\rtimes_\alpha\mathbf{Z}$ or $H\rtimes_\alpha\mathbf{R}$, where  $\alpha(1)$ is a {\em compacting} automorphism of the noncompact group $H$.
\end{thmintro}

We shall give a much more detailled statement in Section~\ref{proofaa}. For now, observe that besides Heintze groups, examples of amenable and non-elementary hyperbolic locally compact groups are provided by the stabilizer of an end in the full automorphism group of a semi-regular locally finite tree. One can also combine a Heintze group with a tree group by some kind of {warped product} construction, which also yields an example of an amenable non-elementary hyperbolic locally compact group. As a result of our analysis, it turns out that \emph{all amenable hyperbolic groups are obtained in this way}. For this more comprehensive description of amenable hyperbolic groups, we refer to Theorem~\ref{thm:main} below. At this point, let us simply mention  the following consequence of that description. 

\begin{thmintro}\label{thmintro:cat-1}
Every amenable hyperbolic  locally compact group acts continuously, properly and cocompactly by isometries on a proper, geodesically complete $\mathrm{CAT}(-1)$ space. 
\end{thmintro}

Those $\mathrm{CAT}(-1)$ spaces will be constructed as fibered products of homogeneous negatively curved manifolds with trees. We call them  \textbf{millefeuille spaces}; see \S\ref{sec:millefeuille} below for a more precise description. Those millefeuille spaces provide rather nice model spaces for amenable hyperbolic groups; one should keep in mind that      for general hyperbolic locally compact groups (even discrete ones), it is an outstanding problem to determine if they can act properly cocompactly on any CAT($-1$) (or even CAT(0)) space \cite[\S 7.B]{Groa}.

\medskip
Another consequence of our study is a complete answer to a  question  appearing at the very end of the paper~\cite{kaimanovich2002boundary} by Kaimanovich and Woess: they asked whether there exists a one-ended locally finite hyperbolic graph with a vertex-transitive group $\Gamma$ of automorphisms fixing a point at infinity. For planar graphs, this was recently settled in the negative by Georgakopoulos and Hamann~\cite{georgakopoulos2011fixing}. We actually show that the answer is negative in full generality.

\begin{corintro}\label{cor:woesskai}
If a locally finite hyperbolic graph admits a vertex-transitive group of automorphisms fixing a point at infinity, then it is quasi-isometric to a regular tree and in particular cannot be one-ended.
\end{corintro}

\medskip
The proof of Theorem~\ref{thm:Amen:intro} can be outlined as follows. If a non-elementary hyperbolic locally compact group $G$ is amenable, it fixes a point in its visual boundary $\bd G$ since otherwise, the ping-pong lemma provides a discrete free subgroup. Notice moreover that since $G$ acts cocompactly on itself, it must contain some hyperbolic isometry. The $G$-action on itself therefore provides a special instance of what we call a \emph{focal action}: namely the action of a group $\Gamma$ on a hyperbolic space $X$ is called \textbf{focal} if $\Gamma$ fixes a boundary point $\xi \in \bd X$ and contains some hyperbolic isometry. If $\Gamma$ fixes a point in $\bd X$ but does not contain any hyperbolic isometry, then the action is called 
\textbf{horocyclic}. Any group admits a horocyclic action on some hyperbolic space, so that not much can be said about the latter type. On the other hand, it is perhaps surprising that focal actions are on the contrary much more restricted: for example any focal action is quasi-convex (see Proposition~\ref{prop:undistorted}). In addition, we shall see how to associate  {\em canonically} a nontrivial {\bf Busemann quasicharacter} $\beta:\Gamma\to\RR$ whenever $\Gamma $ has a focal action on $X$ fixing $\xi \in \bd X$. Roughly speaking, if $b$ is a Busemann function at $\xi$, it satisfies, up to a bounded error, $\beta(g)=b(x)-b(gx)$ for all $(g,x)\in\Gamma\times X$. We refer the reader to \S\ref{hbq} for a rigorous definition. 

Coming back to the setting of Theorem~\ref{thm:Amen:intro}, the amenability of $G$ implies that the {Busemann quasicharacter} is actually a genuine continuous character. The fact that an element which is not annihilated by the character acts as a compacting automorphism on the kernel of that character is finally deduced from an analysis of the dynamics of the boundary action, which concludes the proof of one implication.  

For the converse implication, we give a direct proof that the Cayley graph of a semi-direct product of the requested form is Gromov-hyperbolic. This part of the argument happens to use only metric geometry, without any local compactness assumption. This approach therefore yields a rather general hyperbolicity criterion, which is stated in Theorem~\ref{thm:Contracting} below.

\subsection{Hyperbolic groups with a cocompact amenable subgroup}

We emphasize that, while hyperbolicity is stable under compact extensions, and even under any quasi-isometry, this is not the case for amenability, although amenability is of course invariant under quasi-isometries in the class of discrete groups. Indeed, a noncompact simple Lie group $G$ is  non-amenable but contains a cocompact amenable subgroup, namely the minimal parabolic $P$. The issue is that $G$ is unimodular while $P$ is not, so that $G/P$ does not carry any $G$-invariant measure.  In particular, the class of hyperbolic locally compact groups containing a cocompact amenable subgroup is strictly larger than the class of amenable hyperbolic locally compact groups. The following result shows that there are however not so many non-amenable examples in that class. 

\begin{thmintro}\label{thm:Clight}
Let $G$ be a non-amenable hyperbolic locally compact group. If $G$ contains a cocompact amenable closed subgroup, then $G$ has a unique maximal compact normal subgroup $W$, and exactly one of the following holds:
\begin{enumerate}[(1)]

\item  $G/W$ is the group of isometries or orientation-preserving isometries of a rank one symmetric space of noncompact type.

\item\label{ac} $G/W$ has a continuous, proper, faithful action by automorphisms on a locally finite non-elementary tree $T$, without inversions and with exactly two orbits of vertices, such that the induced $G$-action on the set of ends $\bd T$ is $2$-transitive. In particular, $G/W$ decomposes as a nontrivial amalgam of two profinite groups over a common open subgroup.
\end{enumerate}
\end{thmintro}

A locally compact group is called a \textbf{standard rank one group} if it has no nontrivial compact normal subgroup and satisfies one of the two conditions (1) or (2) in Theorem~\ref{thm:Clight}. 
Standard rank one groups of type (2)  include simple algebraic groups of rank one over non-Archimedean local fields and complete Kac--Moody groups of rank two over finite fields. More exotic examples, and a thorough study in connection with finite primitive groups, are due to Burger and Mozes~\cite{burger2000groups}.

Any standard rank one group contains a cocompact amenable subgroup, namely the stabilizer of a boundary point, so that the converse of Theorem~\ref{thm:Clight} holds as well. In fact, several other characterizations of standard rank one groups are provided by Theorem~\ref{thm:application} below; we shall notably see that they coincide with those noncompact hyperbolic locally compact groups acting transitively on their boundary. 

A consequence of Theorem~\ref{thm:Clight} is that a non-amenable hyperbolic locally compact group which contains a cocompact amenable subgroup is necessarily unimodular. This is a noteworthy fact, since  a non-amenable hyperbolic locally compact groups has no reason to be unimodular in general. For example, consider the HNN extension of $\mathbf{Z}_p$ by the isomorphism between its subgroups $p\mathbf{Z}_p$ and $p^2\mathbf{Z}_p$ given by multiplication by $p$; this group is hyperbolic since it lies as a cocompact subgroup in the automorphism group of a $(p+p^2)$-regular tree, but it is neither amenable nor unimodular.

\subsection{When are non-uniform lattices relatively hyperbolic? }

In a similar way as the concept of Gromov hyperbolic groups  was designed to axiomatize fundamental groups of compact manifolds of negative sectional curvature, relative hyperbolicity was introduced, also by Gromov~\cite{Gro}, to axiomatize fundamental groups of finite volume manifolds of pinched negative curvature. Several equivalent definitions exist in the literature. Let us only recall one of them, which is the most appropriate for our considerations; we refer the reader to rich literature on relative hyperbolicity for other definitions and comparisons between those (the most relevant one for the definition we chose is \cite{Yaman}).  Let $G$ be a locally compact group acting continuously and properly by isometries on a hyperbolic metric space $X$.  Following P.~Tukia~\cite[p.~74]{Tukia},  we say that the $G$-action (or $G$ itself if there is no ambiguity on the action) is \textbf{cusp-uniform} if every boundary point $\xi \in \bd X$ is either a conical limit point or a bounded parabolic point (this notion was introduced by B.~Bowditch~\cite{bowditch1999boundaries}, who called it `\emph{geometrically finite}'). 
The group $G$ is called \textbf{relatively hyperbolic} if it admits some cusp-uniform action on a proper hyperbolic geodesic metric space.


For example, fundamental groups of  finite volume manifolds of pinched negative curvature and, in particular, non-uniform lattices in rank one simple Lie groups, are all relatively hyperbolic: their action on the universal cover of the manifold (resp. on the associated symmetric space) is cusp-uniform. Since rank one simple Lie groups are special instances of hyperbolic locally compact groups, one might expect that  non-uniform lattices in more general  hyperbolic locally compact groups are always relatively hyperbolic. The following result 
shows that this is far from true. 

\begin{thmintro}\label{thm:RelHyp:cocpt}
Let $X$ be a proper hyperbolic geodesic metric space and $G \leq \Isom(X)$ be a closed subgroup acting cocompactly.  

If the action of some  non-cocompact closed subgroup  $\Gamma \leq G$ on $X$ is cusp-uniform, 
then $G$ has a maximal compact normal subgroup $W$ such that $G/W$ is a standard rank one group. 
 \end{thmintro}

Some hyperbolic right-angled buildings, as well as most hyperbolic Kac--Moody buildings, are known to admit non-uniform lattices. Theorem~\ref{thm:RelHyp:cocpt} implies that, provided the building has dimension~$\ge 2$, these lattices are not cusp-uniform (i.e., they are not relatively hyperbolic with respect to the family of        
stabilizers of parabolic points).

As remarked above, a non-uniform lattice in a rank one simple Lie group is always relatively hyperbolic. In the case of tree automorphism groups, this is not the case. Necessary and sufficient conditions for a non-uniform lattice have been described by F.~Paulin~\cite{Paulin}: the key being that a connected fundamental domain for the action of the lattice on the tree has finitely many cusps.

\subsection{Amenable relatively hyperbolic groups}

As we have seen, there are nontrivial examples of locally compact groups that are both amenable and hyperbolic. We may wonder whether even more general examples might be obtained  by considering the class of relatively hyperbolic groups. The following shows that this is in fact not the case. 

\begin{thmintro}\label{thm:RelHyp}
Let $G$ be an amenable locally compact group. If $G$ is relatively hyperbolic, then $G$ is hyperbolic. 
\end{thmintro}

\subsection*{Organization of the paper}

We start with a preliminary section presenting a general construction associating a proper geodesic metric space $X$ to an arbitrary compactly generated locally group $G$, together with a continuous, proper, cocompact $G$-action by isometries. This provides a useful substitute for Cayley graphs, which is better behaved since it avoids the lack of continuity and local compactness that Cayley graphs may have in the non-discrete case. Since  $G$ is quasi-isometric to $X$, the hyperbolicity of the former is equivalent to the hyperbolic of the latter. 

The proofs of the main results are then spread over the rest of the paper, and roughly go into three steps. The first part consists in a general study of isometric actions on hyperbolic spaces, culminating in the proof of Theorem~\ref{thm:Contracting} which implies that certain groups given as semi-direct products with cyclic factor are hyperbolic. This part is mostly developed in a purely metric set-up, without the assumption of local compactness. It occupies  Section~\ref{sec:AmenableActions} and Section~\ref{sec:Focal}, and yields the implication from right to left in Theorem~\ref{thm:Amen:intro}.

In Section~\ref{sec:ProperActions}, we start making local compactness assumptions, but yet not exploiting any deep structural result about locally compact groups. This is where Theorem~\ref{thm:RelHyp} is proven. 
This chapter moreover provides the implication from left to right in Theorem~\ref{thm:Amen:intro}, whose proof is thus completed in Section~\ref{sec:ProofThmA}.

Finally, a more comprehensive version of Theorem~\ref{thm:Amen:intro}, as well as its corollaries, is proved in \S\ref{proofaa}
after some preliminary work about the structure of groups admitting compacting automorphisms in Section~\ref{sec:compaction}, and on the construction of millefeuille spaces in \S \ref{sec:millefeuille}. Similarly, a more comprehensive version of Theorem~\ref{thm:Clight} is stated and proved in \S \ref{proofb}. Theorem~\ref{thm:RelHyp:cocpt} is then easily deduced in the next subsection.


\subsection*{Acknowledgement} We thank the referee, whose comments were helpful in improving the presentation of the paper.

\section{Preliminaries on geodesic spaces for locally compact groups}\label{sec:pre}

It is well-known that a topological group with a proper, cocompact action by isometries on a locally compact geodesic metric space is necessarily locally compact and compactly generated. It turns out that the converse is true, and that the space can be chosen to be a piecewise-manifold. This is the content of Proposition~\ref{prop:Cayley} below. Its relevance to the rest of the paper is through Corollary~\ref{cor:eqprop}. The remainder of the section is devoted to its proof and is independent from the rest of the paper, so the reader can, in a first reading, take the proposition and its corollary for granted and go directly to Section~\ref{sec:AmenableActions}.

\begin{prop}\label{prop:Cayley}
Let $G$ be a compactly generated, locally compact group. There exists a finite-dimensional (in the sense of topological dimension) locally compact geodesic metric space $X$ with a continuous, proper, cocompact $G$-action by isometries.

In fact $X$ is a connected locally finite gluing of Riemannian manifolds along their boundaries.
\end{prop}

An immediate consequence is the fact that a closed cocompact subgroup of a compactly generated locally compact group is itself compactly generated. This is well-known and can be established more simply by a direct algebraic argument, see \cite{MS59}.

\medskip
Let us begin by illustrating Proposition~\ref{prop:Cayley} with significant examples.
\begin{itemize}
\item When $G$ is discrete, we consider its Cayley graph with respect to a finite generating set.
\item When $G$ is a connected Lie group, $X$ is taken as $G$ endowed with a left-invariant Riemannian metric.
\item When $G$ is an arbitrary Lie group, we pick a finite subset $S$ whose image in $G/G^\circ$ is a generating subset and endow the (non-connected) manifold $G$ with a left-invariant Riemannian metric; then for each coset $L$ of $G^\circ$ and each $s\in S$, we consider a strip $L\times [0,1]$ (with the product Riemannian metric) and attach it to $G$ by identifying $(g,0)$ to $g$ and $(g,1)$ to $gs$. The resulting space $G$ is path-connected and endowed with the inner length metric associated to the Riemannian metric on each strip.
\item When $G$ is totally disconnected, a Cayley--Abels graph construction is available, generalizing the discrete case. It goes back to Abels \cite[Beispiel~5.2]{abels1974specker}.
It consists in picking a compact generating subset $S$ that is bi-invariant under the action of some compact open subgroup $K$, considering the (oriented, but unlabeled) Cayley graph of $G$ with respect to $S$, and modding out by the right action of $K$. The resulting graph is locally finite; the action of $G$ is continuous, vertex-transitive and proper, the stabilizer of the base-vertex is $K$.
\end{itemize}


The general case is a common denominator between the latter two  constructions.
Roughly speaking, we construct $X$ by fibering a Lie quotient associated to $G^\circ$ over a Cayley--Abels graph for $G/G^\circ$. The following classical theorem allows to bypass some of the technical difficulties.  

\begin{thm}[H.~Yamabe] \label{thm:Yamabe}
Let $G$ be a connected-by-compact locally compact group. Then $G$ is compact-by-(virtually connected Lie).
\end{thm}
(By convention (A)-by-(B) means with a normal subgroup satisfying (A) so that the quotient group satisfies (B).)
\begin{proof}
See Theorem~4.6 in~\cite{montgomery1955topological}.
\end{proof}

\begin{lem}\label{lem:open-image}
Let $G$ be any locally compact group. There is a compact subgroup $K$ whose image in $G/G^\circ$ is open. If $G^\circ$ is Lie, then we can assume that $K\cap G^\circ=1$.
\end{lem}

\begin{proof}[Proof of Lemma~\ref{lem:open-image}]
By van Dantzig's theorem~\cite[p.~18]{vanDantzig31}, the totally disconnected  group $G/G^\circ$ contains a compact open subgroup; let $H$ be the pre-image in $G$ of that subgroup, so that $H$ is open in $G$ and $H/H^\circ$ is compact.

Thus Yamabe's theorem  applies and $H$ contains a compact normal subgroup $K$ such that $H/K$ is a Lie group. Since $H/H^\circ$ is compact, it follows that $H/K$ has finitely many connected components and therefore, upon replacing $H$ by a smaller open subgroup, we can assume that $H/K$ is a connected Lie group. The next lemma, of independent interest, implies in particular that we have $H^\circ.K=H$; thus indeed the image of $K$ in $G/G^\circ$ is the open subgroup $H/H^\circ$.

For the additional statement, assume that $G^\circ$ is Lie. Then there exists a neighbourhood $V$ of 1 such that $V\cap G^\circ$ contains no nontrivial subgroup. Since any compact group is pro-Lie (by Peter-Weyl's Theorem), $V$ contains a normal subgroup $K'$ of $K$ such that $K/K'$ is Lie; clearly we have $K'\cap G^\circ  =1$. If $\pi$ is the projection to $G/G^\circ$, it follows that $\pi(K)/\pi(K')$ is both Lie and profinite, hence is finite, so $\pi(K')$ is open in $G/G^\circ$ as well. 
\end{proof}

\begin{lem}\label{lem:quot-Lie}
Let $G$ be a locally compact group with a quotient map $\pi \colon G\twoheadrightarrow L$ onto a Lie group $L$. Then $\pi(G^\circ)=L^\circ$.
\end{lem}

\begin{proof}
Obviously (considering $\pi^{-1}(L^\circ)\to L^\circ$) we can suppose that $L$ is connected.
We have to show that $NG^\circ=G$, where $N$ is the kernel of $\pi$. Since $G/\overline{NG^\circ}$ is both a quotient of the connected group $L=G/N$ and of the totally disconnected group $G/G^\circ$, it is trivial, in particular $NG^\circ$ is dense. This allows to conclude at least when $N$ is compact or $G$ is a Lie group (not assumed connected). Indeed, in both cases this assumption implies that $NG^\circ$ is closed (when $G$ is a Lie group, see this by modding out by $N\cap G^\circ$).

In general, let $\Omega$ be an open subgroup of $G$ such that $\Omega/G^\circ$ is compact. Since $\pi$ is an open map, $\pi(\Omega)$ is an open subgroup of $L$ and therefore is equal to $L$. This allows to assume that $G$ is connected-by-compact. By Yamabe's theorem, $G$ thus has a maximal compact normal subgroup $W$. We can factor $\pi$ as the composition of two quotient maps $G\to G/(N\cap W)\to G/N=L$. The left-hand map has compact kernel, and $G/(N\cap W)$ has a continuous injective map into $G/N\times G/W$ and therefore is a Lie group. So the result follows from the two special cases above.
\end{proof}

\begin{remark}
The assumption that $L$ is a Lie group is essential in Lemma~\ref{lem:quot-Lie}. Indeed, let $G = \RR \times \ZZ_p$, where $\ZZ_p$ denotes the (compact) additive group of the $p$-adic integers. Let $Z$ be a copy of $\ZZ$ embedded diagonally in $G$, and let $L = G/Z$ be the quotient group. The group $L$ is the so-called \textbf{solenoid} and can alternatively be defined as the inverse limit of the iterated $p$-fold covers of the circle group. It is connected (but not locally arcwise connected). The image of $G^\circ = \RR$ under the quotient map $\pi : G \to L$ is dense, but properly contained, in $L$.
\end{remark}

\begin{proof}[Proof of Proposition~\ref{prop:Cayley}]
Let $G$ be a compactly generated locally compact group. Upon modding out $G$ by the unique maximal compact normal subgroup of $G^\circ$, we can assume that $G^\circ$ is a Lie group and   endow it with a left-invariant Riemannian metric. Set $Q=G/G^\circ$. By Lemma~\ref{lem:open-image}, there is a compact subgroup $U<G$ whose image $W$ in $Q$ is open and $U\cap G^\circ=1$.

Let $S\se Q$ be a compact generating set with $S=S\inv=WSW$. Since $W$ is open, $W\backslash S/W$ is finite and we pick a finite set $\{z_1, \ldots, z_r\}$ of representatives $z_i\in S$. We define for $1=1, \ldots, r$
$$W_i :=W\cap z_i W z_i\inv, \kern5mm L_i := U G^\circ \cap z_i U G^\circ  z_i\inv, \kern5mm V_i := L_i \cap U \kern5mm X_i:=G/V_i \times [0,1].$$
We recall that a Cayley--Abels graph for $Q$ is given by the discrete vertex set $Q/W$ and the $r$ (oriented edge)-sets $Q/W_i$, with natural $Q$-action and source and target maps. On the other hand, there are canonical surjective $G$-maps $X_i\twoheadrightarrow Q/W_i$. The fibres of these maps are $L_i/V_i \times [0,1] \cong G^\circ  \times [0,1]$, which are indeed connected Riemannian manifolds. By construction, the $G$-action is compatible with the gluings of boundary components of $X_i$ determined by the map to the Cayley--Abels graph, and we endow the resulting connected space $X$ with the inner length metric associated to the Riemannian structure. Explicitly, the gluing is generated by identifying $(g V_i, 0)$ with $(h V_j, 0)$ whenever $gU=hU$ and $(g V_i, 1)$ with $(h V_i, 0)$ whenever $g z_i U = hU$.
\end{proof}

The above construction is of course much more general than what is needed in the present article.
Indeed, as a consequence of the results of the article, a hyperbolic locally compact group has a continuous proper cocompact action either on a millefeuille space (this includes the special case when this space is simply a homogeneous negatively curved manifold), or on a connected graph. 
Actually, the latter description shows that the space can be chosen in addition to be contractible: indeed, in the case of a connected graph, the Rips complex construction as described in~\cite[1.7.A]{Gro} is applicable.

For the time being, we only record the following consequence of Proposition~\ref{prop:Cayley}.

\begin{cor}\label{cor:eqprop}
For a  locally compact group $G$, the following are equivalent.
\begin{enumerate}[(i)]
\item $G$ is hyperbolic, i.e.\ compactly generated and word hyperbolic with respect to some compact generating set.
\item $G$ has a continuous proper cocompact isometric action on a proper geodesic hyperbolic space.\qed
\end{enumerate}
\end{cor}

\section{Actions on hyperbolic spaces}\label{sec:AmenableActions}

\selectlanguage{french}
\begin{flushright}
\begin{minipage}[t]{0.7\linewidth}
\small\itshape  {\og \c Ca faut avouer, dit Trouscaillon qui, dans cette simple ellipse, utilisait hyperboliquement le cercle vicieux de la parabole.\fg}
\upshape
\begin{flushright}
(R.~Queneau, Zazie dans le m\'etro, 1959)
\end{flushright}
\end{minipage}
\end{flushright}
\selectlanguage{english}  

\fixbug

\medskip

After reviewing some basic features of groups acting on hyperbolic spaces, the goal of this section is to highlight the importance of  \emph{focal} actions   (see Section~\ref{sec:Gromov_classif} below for the precise definitions). Indeed, while actions of general type have been studied in thorough detail in a myriad of papers on hyperbolic spaces, other actions, sometimes termed as ``elementary", have been considered as uninteresting. Notably, and as a consequence of an inadequate terminology, the distinction between \emph{horocyclic} and \emph{focal} actions has been eclipsed. Several basic results  in this section (especially Proposition~\ref{prop:undistorted}, Lemma~\ref{lem:focaltt}, and Proposition~\ref{prop:coctamen}) illustrate how different these two types of actions are and how {essential it is to take} a specific look at focal actions.

Throughout this section, we let $X$ be a Gromov-hyperbolic geodesic metric space.

Recall that $X$ is called \textbf{proper} if closed balls are compact; due to the Hopf--Rinow theorem for length spaces, it is equivalent to require that $X$ be locally compact and complete. Recall further that the full isometry group $\Isom(X)$, endowed with the compact open topology, is a second countable locally compact group. We emphasize that $X$ will \emph{not} be assumed proper, unless explicitly stated otherwise.

\subsection{Gromov's classification}\label{sec:Gromov_classif}

The material in this section follows  from \cite[3.1]{Gro}.
Let $\Gamma$ be an abstract group, and consider an arbitrary isometric action $\alpha$ of $\Gamma$ on a nonempty hyperbolic geodesic metric space $X$.

The {\bf visual boundary} (or {\bf boundary}) $\bd X$ of $X$ is defined as follows. Fix a basepoint $x$ in $X$, define the norm $|y|_x=d(x,y)$ and 
the Gromov product $$2(y|z)_x=|y|_x+|z|_x-d(y,z).$$ Note that $|(y|z)_x-(y|z)_w|\le d(x,w).$ A sequence $(x_n)$ in $X$ is {\bf Cauchy-Gromov} if $(x_n|x_m)_x$ tends to infinity when $n,m$ both tend to $\infty$; by the previous inequality, this does not depend on the choice of $x$. We identify two Cauchy-Gromov sequences $(y_n)$ and $(z_n)$ if $(y_n|z_n)$ tends to infinity. This is indeed an equivalence relation if $X$ is $\delta$-hyperbolic, in view of the inequality
$$(y|z)_x\ge\min[(y|w)_x,(w|z)_x]-\delta,$$
whose validity is a definition of $\delta$-hyperbolicity (if it holds for all $x,y,z,w$).
The boundary $\bd X$ is the quotient set of Cauchy-Gromov sequences by this equivalence relation. (In other words, consider the uniform structure given by the entourages $\{(y,z) : (y|z)_x \geq r\}$ as $r$ ranges over $\mathbf{R}$; then $\bd X$ is the completion from which the canonical image of $X$ has been removed.)

If $\Gamma$ is a group acting on $X$ by isometries, the \textbf{boundary} $\bd_X\Gamma$, also called the \textbf{limit set} of $\Gamma$ in $X$, consists of those elements $(y_n)$ in the boundary, that can be represented by a Cauchy-Gromov sequence of the form $(g_nx)$ with $g_n\in \Gamma$. Since $(g_nx)$ and $(g_nw)$ are equivalent for all $x,w$, this does not depend on $x$. The action of $\Gamma$ induces an action on $\bd X$, which preserves the subset $\bd_X \Gamma$. 

A crucial case is when $\Gamma$ is generated by one isometry $\phi$. Recall that $\phi$ is called
\begin{itemize}
\item[{\bf --~elliptic}] if it has bounded orbits;
\item[{\bf --~parabolic}] if it has unbounded orbits and $\lim_{n\to\infty}\frac1n|\phi^n(x)|_x=0$; 
\item[{\bf --~hyperbolic}] if $\lim_{n\to\infty}\frac1n|\phi^n(x)|_x>0$.
\end{itemize}
The above limit always exists by subadditivity, and the definition clearly does not depend on the choice of $x$. Also, it is straightforward that if $\phi$ preserves a geodesic subset $Y$, then the type of $\phi|_Y$ is the same as the type of $\phi$. In terms of boundary, it can be checked \cite[Chap.~9]{coornaert1990geometrie} that 
\begin{itemize}
\item $\phi$ is elliptic $\Leftrightarrow$ $\bd_X\langle\phi\rangle$ is empty;
\item $\phi$ is parabolic $\Leftrightarrow$ $\bd_X\langle\phi\rangle$ is a singleton;
\item $\phi$ is hyperbolic $\Leftrightarrow$ $\bd_X\langle\phi\rangle$ consists of exactly two points.
\end{itemize}

For an action of an arbitrary group $\Gamma$, Gromov's classification \cite[3.1]{Gro} goes as follows. The action is called 
\begin{itemize}
\item {\bf elementary} and
\begin{itemize}
\item[{\bf -- bounded}] if orbits are bounded;
\item[{\bf -- horocyclic}] if it is unbounded and has no hyperbolic element;
\item[{\bf -- lineal}] if it has a hyperbolic element and any two hyperbolic elements have the same endpoints;
\end{itemize}
\item {\bf non-elementary}\footnote{We follow Gromov's convention. It turns out that in the special case of proper actions of discrete groups, focal actions do not exist and thus elementary actions are precisely those with a finite orbit on the boundary. For this reason, Gromov's conventions were misinterpreted by several authors, who unaccurately consider the focal case as elementary.} and
\begin{itemize}
\item[{\bf -- focal}] if it has a hyperbolic element, is not lineal and any two hyperbolic elements have a common endpoint (it easily follows that there is a common endpoint for all hyperbolic elements);
\item[{\bf -- general type}] if it has two hyperbolic elements with no common endpoint.
\end{itemize}
\end{itemize}

These conditions can be described in terms of the boundary $\bd_X \Gamma$.
\begin{prop}\label{prop:types}
The action of $\Gamma$ is 
\begin{itemize}
\item {\bf bounded} if and only if $\bd_X \Gamma$ is empty;
\item {\bf horocyclic} if and only if $\bd_X \Gamma$ is reduced to one point; then $\bd_X\Gamma$ is the unique finite orbit of $\Gamma$ in $\bd X$;
\item {\bf lineal} if and only if $\bd_X \Gamma$ consists of two points; then $\bd_X\Gamma$ contains all finite orbits of $\Gamma$ in $\bd X$;
\item {\bf focal} if and only if $\bd_X \Gamma$ is uncountable and $\Gamma$ has a fixed point $\xi$ in $\bd_X\Gamma$; then $\{\xi\}$ is the unique finite orbit of $\Gamma$ in $\bd X$;
\item {\bf of general type} if and only if $\bd_X \Gamma$ is uncountable and $\Gamma$ has no finite orbit in $\bd X$.
\end{itemize}
In particular, the action is elementary if and only if $\bd_X\Gamma$ has at most two elements, and otherwise $\bd_X\Gamma$ is uncountable.
\end{prop}

\begin{proof}[Sketch of proof] 
If the action is horocyclic, the proof of \cite[Theorem 9.2.1]{coornaert1990geometrie} shows that for every sequence $(g_n)$ such that $|g_nx|_x$ tends to infinity, the sequence $(g_nx)$ is Cauchy-Gromov; it follows that $\bd_X\Gamma$ is a singleton. It follows in particular that the intersection of an orbit with any quasi-geodesic is bounded. 

If $\bd_X\Gamma=\{\xi\}$ and $\Gamma$ has another finite orbit on the boundary, then we can suppose that it has another fixed point $\eta$ by passing to a subgroup of finite index.  
Let us consider a (metric) ultrapower $X^*$ of $X$; namely $X^*$ is obtained as follows:  
endow the space of bounded sequences in $X$ with the pseudo-distance defined as the limit of the distances along a non-principal ultrafilter; then $X^*$ is the metric space obtained by identifying sequences at pseudo-distance zero.
%
It admits a canonical isometric embedding of $X$; it is also a geodesic metric space and is hyperbolic with the same hyperbolicity constant, and the $\Gamma$-action canonically extends to an action $X^*$. There is an natural inclusion $\partial X\subset\partial X^\ast$; since $X$ is $\Gamma$-invariant, it follows that $\partial_X\Gamma=\partial_{X^\ast}\Gamma$ and in particular, the type of the action on $X^*$ is the same as the type of the action on $X$, i.e., horocyclic. Moreover, any pair of distinct points in $\partial X$ can be joined by a geodesic in $X^*$. Consider a geodesic in $X^*$ joining $\xi$ and $\eta$. Its $\Gamma$-orbit is a $\Gamma$-invariant quasi-geodesic.
Since the action is horocyclic, the above remark shows that the action of $\Gamma$ on this quasi-geodesic, and hence on $X$, is bounded, a contradiction.

The other verifications are left to the reader (the uncountability of $\bd_X \Gamma$ in the non-elementary cases follows from Lemma~\ref{lem:schottky}).
\end{proof}

\subsection{Basic properties of actions and quasi-convexity}

As before, $X$ is a hyperbolic geodesic space, without properness assumptions.
Recall that a subset $Y\subset X$ is {\bf quasi-convex} if there exists $c\ge 0$ such that for all $x,y\in Y$ there exist a sequence $x=x_0,\dots,x_n=y$ in $Y$ with $d(x_i,x_{i+1})\le c$ for all $i$ and $n\le c(d(x,y)+1)$. 
We say that an action is {\bf quasi-convex} if some (and hence every) orbit is quasi-convex. If the acting group $\Gamma$ is locally compact and the action is metrically proper, this is equivalent to the requirement that $\Gamma$ is compactly generated and undistorted in $X$ (i.e.\ the orbit map $g\mapsto gx$ is a quasi-isometric embedding for some/all $x\in X$).

The notions of horocyclic and focal actions, i.e.\ those unbounded actions with a unique fixed point at infinity, are gathered under the term of {\em quasi-parabolic} actions in~\cite{Gro,karlsson2004some}, while horocyclic actions were termed {\em parabolic}.
 Nevertheless, the following proposition, which does not seem to appear in the literature, shows that horocyclic and focal actions exhibit a dramatically opposite behaviour.

\begin{prop}\label{prop:undistorted} 
If the action of $\Gamma$ is bounded, lineal or focal, then it is quasi-convex. On the other hand, a horocyclic action is never quasi-convex, while an action of general type can be either quasi-convex or not.
\end{prop}
\begin{proof}
The bounded case is trivial and in the lineal case, $\Gamma$ preserves a subset at bounded Hausdorff distance of a geodesic and is thus quasi-convex.

Assume that the action is focal with $\xi \in \bd X$ as a global fixed point. We have to prove that some given orbit is quasi-convex. Let $\alpha$ be a hyperbolic element,       
and let $x_0$ be a point on a geodesic line $[\xi,\eta]$ between the two        
fixed points $\xi$ and $\eta$ of $\alpha$ (embed if necessary $X$ into a (metric) ultrapower $X^*$ as in the proof of Proposition \ref{prop:types} to ensure the existence of this geodesic). The $\langle\alpha\rangle$-orbit     
of $x_0$ is a discrete quasi-geodesic. Observe that the orbit $\Gamma x_0$ is the     
union of quasi-geodesics $g\langle\alpha\rangle x_0$, with $g$ varying in       
$\Gamma$. In particular, $\Gamma x_0$ contains quasi-geodesics between all its points      
to $\xi$.                                                                       
 Now, let $x$ and $y$ be two points in $\Gamma x_0$.                                  
{Recall} that given a quasi-geodesic triangle between three points in the reunion of a          
hyperbolic space with its boundary, the union of two edges of this triangle    
is quasi-convex. Applying this to $x$, $y$ and $\xi$, we see that a          
quasi-geodesic between $x$ and $y$ can be found in the orbit $\Gamma x_0$, which      
is therefore quasi-convex.

If the action of $\Gamma$ is horocyclic, we observed in the proof of Proposition~\ref{prop:types} that the intersection of any orbit with a quasi-geodesic is bounded. More precisely, for every $c$ there exists $c'$ (depending only on $\delta$ and $c$) such that the intersection of any orbit and any $c$-quasi-geodesic is contained in the union of two $c'$-balls. If the action is quasi-convex, given $x$ there exists $c$ such that any two points in $\Gamma x$ are joined by a $c$-quasi-geodesic within $\Gamma x$; taking two points at distance $>2c'+c$ we obtain a contradiction.

For the last statement, it suffices to exhibit classical examples: for instance $\textnormal{SL}_2(\ZZ)$ has a proper cocompact action on a tree, but its action on the hyperbolic plane is not quasi-convex.
\end{proof}

Let $\Gamma$ act on $X$ by isometries. Recall that a {\bf Schottky subsemigroup}, resp.\ {\bf subgroup}, for the action of $\Gamma$ on $X$ is a pair $(a,b)$ such that the orbit map $g\mapsto gx$, is a quasi-isometric embedding of the free semigroup (resp.\ subgroup) on $(a,b)$. An elementary application of the ping-pong lemma \cite[8.2.E, 8.2.F]{Gro} yields the following.

\begin{lem}\label{lem:schottky}
If the action of $\Gamma$ is focal (resp.\ general type), then there is a Schottky subsemigroup (resp.\ subgroup) for the action of $\Gamma$ on $X$.\qed
\end{lem}

It is useful to use a (metrizable) topology on $\bd X$. A basis of neighbourhoods of the boundary point represented by the Cauchy-Gromov sequence $(x_i)$ is 
$$V_n=\{(y_i):\;\liminf (y_i|x_i)\ge n\}.$$

Gromov shows \cite[8.2.H]{Gro} that if the action is of general type, then the actions of $\Gamma$ on $\bd_X\Gamma$ and on $\bd_X\Gamma\times \bd_X\Gamma$ are topologically transitive. This very important (and classical) fact will not be used in the paper. On the other hand, in the focal case, we have the following observation, which, as far as we know, is original.

\begin{lem}\label{lem:focaltt}
Let the isometric action of $\Gamma$ on $X$ be focal with $\xi \in \bd X$ as a fixed point. Then the action of $\Gamma$ on $\bd_X\Gamma-\{\xi\}$ is topologically transitive.
\end{lem}
\begin{proof}
By Proposition~\ref{prop:undistorted}, there is no loss of generality in assuming that the action of $\Gamma$ on $X$ is cobounded, so that $X$ is covered by $r$-balls around points of an orbit $\Gamma x$ for some $r<\infty$. Fix a point $\eta$ and an open subset $\Omega$ in $\bd X-\{\xi\}$. For some $c$ there exists a $c$-quasi-geodesic $D$ in $\Gamma x$ joining $\xi$ and $\eta$; let $gx$ be one of its points. There exists a $r$-ball $B$ so that every $c$-quasi-geodesic with endpoint $\xi$ and passing through $B$ has its second endpoint in $\Omega$. There exists $h\in\Gamma$ such that $hgx$ belongs to $B$. It follows that the second endpoint of $hD$, which equals $h\eta$, lies in $\Omega$. 
\end{proof}

\begin{lem}\label{lem:properU}
Let the isometric action of $\Gamma$ on $X$ be horocyclic with fixed point $\xi$ in $\bd X$. Fix $x\in X$ and endow $\Gamma$ with the left-invariant (pseudo)distance $d_x(g,h)=d(gx,hx)$. 

Then the action of $\Gamma$ on $\bd X-\{\xi\}$ satisfies the following property (akin to metric properness):
 for every closed subset $K$ of $\bd X$ not containing $\xi$, the set $\{g\in\Gamma:gK\cap K\neq\varnothing\}$ is bounded in $(\Gamma,d_x)$.
\end{lem}
\begin{proof}
We can replace $X$ by a (metric) ultrapower (see the proof of Proposition \ref{prop:types}), allowing the existence of geodesics between any two points (at infinity or not).

Let $K$ be a closed subset of $\bd X$ not containing $\xi$. Note that there exists a ball $B$ of radius $10\delta$ such that every geodesic whose endpoints are $\xi$ and some point in $K$, passes through $B$. Fix $g\in\Gamma$ such that $gK\cap K\neq\varnothing$. In particular, there exists a geodesic $D$ such that both $D$ and $gD$ are issued from $\xi$ and pass through $B$. It follows that before hitting $B$, $D$ and $gD$ lie at bounded distance (say, $\le 50\delta$) from each other. Fix some $y\in D\cap B$ and let $D_y$ be the geodesic ray in $D$ joining $y$ to $\xi$. Then either $gD_y$ is contained in the $50\delta$-neighbourhood of $D_y$, or vice versa. Since $g$ is not hyperbolic, using the inequality $d(g^2y,y)\le d(gy,y)+2\delta$ (see \cite[Lemma~9.2.2]{coornaert1990geometrie}) it easily follows (using that $y$, $gy$ and $g^2y$ are close to the geodesic ray $D_y$ and the equality $d(y,gy)=d(gy,g^2y)$) that $d(y,gy)\le 152\delta$. Thus the set $\{g\in\Gamma; gK\cap K\neq\varnothing\}$ is $d_y$-bounded, hence $d_x$-bounded.
\end{proof}

Define the {\bf bounded radical} of $\Gamma$ as the union $B_X(\Gamma)$ of all normal subgroups $N$ that are \textbf{$X$-bounded}, i.e.\  such that the action of $N$ on $X$ is bounded.

\begin{lem}\label{lem:bdrad}Suppose that the $\Gamma$-action is not horocyclic. Then the following properties hold.
\begin{enumerate}[(a)]
\item\label{bdr} The action of $B_X(\Gamma)$ on $X$ is bounded. 
\item\label{nli} If moreover the $\Gamma$-action is neither lineal, $B_X(\Gamma)$ is equal to the kernel $K$ of the action of $\Gamma$ on $\bd_X\Gamma$.
\end{enumerate}
\end{lem}

\begin{proof}
If $N$ is an $X$-bounded normal subgroup, then $Nx$ is bounded, and therefore the $Ny$ are uniformly bounded when $y$ ranges over an orbit $\Gamma x$. In particular, the action of $N$ on $\bd_X\Gamma$ is trivial. This proves the inclusion $B_X(\Gamma)\subset K$ (without assumption on the action).

Note that the action of $K$ on $\bd_XK$ is trivial, so by Proposition~\ref{prop:types}, the action of $K$ is bounded, lineal or horocyclic; Lemma~\ref{lem:properU} then shows that the action of $K$ cannot be horocyclic.
\begin{itemize}
\item If the action of $K$ on $\bd_XK$ is bounded, it follows from the definition of $B_X(\Gamma)$ that $K\subset B_X(\Gamma)$, and both (\ref{bdr}) and (\ref{nli}) follow.

\item If the action of $K$ on $\bd_XK$ is lineal, its 2-element boundary is preserved by $\Gamma$ and therefore the action of $\Gamma$ is lineal as well (so we do have to consider (\ref{nli}). Since $B_X(\Gamma)$ consists of elliptic isometries, its action is either bounded or horocyclic, but since a horocyclic action cannot preserve a 2-element subset in $\bd_X\Gamma$ by Proposition~\ref{prop:types}, the action of $B_X(\Gamma)\subset K$ cannot be horocyclic and therefore is bounded, so (\ref{bdr}) is proved.
\end{itemize}
\end{proof}

\subsection{Horofunctions and the Busemann quasicharacter}\label{hbq}

The material in this section is based on \cite[7.5.D]{Gro} and \cite[Sec.~4]{manning}. 

Let $X$ be an arbitrary hyperbolic space 
and $\xi \in \bd X$ be a point at infinity. We define a \textbf{horokernel} based at $\xi$ to be any accumulation point (in the topology of pointwise convergence) of a sequence of functions
$$X\times X\lra \RR, \kern10mm (x,y) \longmapsto d(x, x_n) - d(y, x_n),$$
where $\{x_n\}$ is any sequence in $X$ converging to $\xi$, i.e.\ a Cauchy-Gromov sequence representing $\xi$. By the Tychonoff theorem, the collection $\mathcal{H}_\xi$ of all horokernels based at $\xi$ is non-empty; it consists of continuous functions, indeed $1$-Lipschitz in each variable. Moreover, any horokernel is antisymmetric by definition. For definiteness, many authors propose the gordian definition of the \textbf{Busemann kernel} $b_\xi$ of $\xi$ as the supremum of $\mathcal{H}_\xi$ (losing continuity and antisymmetry in general); it turns out that it remains at bounded distance of any horokernel, the bound depending only on the hyperbolicity constant of $X$, see~\S\,8 in~\cite{GhysHarpe}. Let us also mention the notion of {\bf horofunction} $h'(x)=h(x,x_0)$, which depends on the choice of a basepoint $x_0$.

\medskip
Recall that a function $f:\Gamma\to \RR$ defined on a group $\Gamma$ is a \textbf{quasicharacter} (also known as quasimorphism) if the \textbf{defect}
$$\sup_{g, h\in G} \Big|f(g) + f(h) - f(gh) \Big|$$
is finite; it is called \textbf{homogeneous} if moreover $f(g^n)=nf(g)$ for all $g\in G$ and $n\in\ZZ$; in that case, $f$ is constant on conjugacy classes. Given an isometric group action on $X$ fixing $\xi$, there is a canonical homogeneous quasicharacter associated to the action, which was constructed by J.~Manning \cite[Sec.~4]{manning}. The following is a variant of an idea appearing in T.~B\"uhler's (unpublished) Master's thesis.

\begin{prop}\label{prop:Buse-char}
Let $G$ be a locally compact group acting continuously by isometries on $X$. Let $\xi\in \partial X$, $h\in \mathcal{H}_\xi$ and $x\in X$. Then the function
$$\beta_\xi: G_\xi \lra \RR,\kern10mm \beta_\xi(g)= \lim_{n\to\infty} \frac 1n h(x, g^n x)$$
is a well-defined continuous homogeneous quasicharacter, called {\bf Busemann quasicharacter} of $G_\xi$, and is independent of $h\in \mathcal{H}_\xi$ and of $x\in X$.

Moreover, the differences $\beta_\xi(g) - h(x, g x)$ and $\beta_\xi(g) - b_\xi(x, g x)$ are bounded (the bound depending only on the hyperbolicity constant of $X$).
\end{prop}

\begin{proof}
By a direct computation, we have
$$h(x,g_1g_2x)-h(x,g_1x)-h(x,g_2x)=h(g_1x,g_1g_2x)-h(x,g_2x);$$
since $\xi$ is fixed by $g_1$, the latter quantity is bounded by a constant depending only on the hyperbolicity of $X$.
Therefore, the function $g\mapsto h(x, g x)$ is a continuous quasicharacter.
Given any quasicharacter $f$ on a group $G$, it is well-known that for all $g$, the sequence $f(g^n)/n$ converges (because the sequence $\{f(g^n)-c\}_n$ is subadditive, where the constant $c$ is the defect) and that the limit is a homogeneous quasicharacter (by an elementary verification). This limit is the unique homogeneous quasicharacter at bounded distance from $f$ and a bounded perturbation of $f$ yields the same limit. Returning to our situation, it only remains to justify that the limit is continuous. It is Borel by definition, and any Borel homogeneous quasicharacter on a locally compact group is continuous~\cite[7.4]{BIW}.
\end{proof}

The Busemann quasicharacter is useful for some very basic analysis of boundary dynamics:

\begin{lem}\label{lem:busatt}
Let $\Gamma$ act on $X$ by isometries and let $\xi$ be a boundary point. Then the (possibly empty) set of hyperbolic isometries in $\Gamma_\xi$ is $\{g\in \Gamma_\xi: \beta_\xi(g)\neq 0\}$, and the set of those with attracting fixed point $\xi$ is $\{g\in \Gamma_\xi: \beta_\xi(g)>0\}$. In particular, the action of $\Gamma_\xi$ is bounded/horocyclic if and only if $\beta_\xi=0$, and lineal/focal otherwise.
\end{lem}

\begin{proof}
Elements $g$ of $\Gamma_\xi$ acting as hyperbolic isometries with attracting fixed point $\xi$ satisfy $\beta_\xi(g)>0$, as we see by direct comparison with horokernels. In particular all elements $g$ acting as hyperbolic isometries satisfy $\beta_\xi(g)\neq 0$. 

Conversely, if $g\in \Gamma_\xi$ satisfies $\beta_\xi(g)>0$, then $\beta_\xi(g^n)$ being linear in $n$, the sequence $\{g^n x : n\in\ZZ\}$ is a quasi-geodesic with the $+\infty$-endpoint at $\xi$. By hyperbolicity of $X$, it follows that $g$ is a hyperbolic isometry with attracting fixed point $\xi$. It also follows that if $g\in\Gamma_\xi$ satisfies $\beta_\xi(g)\neq 0$ then it acts as a hyperbolic isometry.
\end{proof}

The Busemann quasicharacter $\beta_\xi$ is particularly nice in connection with amenability: indeed, as a corollary of Proposition~\ref{prop:Buse-char}, we deduce the following. 

\begin{cor}\label{cor:Buse-char}
Let $G$ be a locally compact group acting continuously by isometries on $X$ and fixing the boundary point $\xi \in \bd X$. Assume that $G$ is amenable, or that $X$ is proper.

Then the Busemann quasicharacter $\beta_\xi: G \to \RR$ is a continuous group homomorphism (then called {\bf Busemann character}) at bounded distance of $g\mapsto b_\xi(x, g x)$ independently of $x\in X$.
\end{cor}

\begin{proof}[Proof of Corollary~\ref{cor:Buse-char}]
The well-known fact that a homogeneous quasicharacter $f$ of an amenable group $M$ is a homomorphism can be verified explicitly by observing that one has $f(g) = m\big(h\mapsto f(gh) - f(h)\big)$ when $m$ is an invariant mean on $M$.

If $G$ is amenable, this applies directly. Now assume that $X$ is proper. If $\beta_\xi=0$ the result is trivial, so assume that the action is lineal/focal. Since $X$ is proper, we can suppose that $G=\Isom(X)_\xi$ (and thus acts properly) and we argue as follows. By Proposition~\ref{prop:undistorted}, sufficiently large bounded neighbourhoods of a  $G$-orbit  are quasi-geodesic. Therefore, replacing $X$ by such a subset if necessary, we can assume that the $G$-action on $X$ is cocompact, so that Lemma~\ref{lem:AmyStab} below applies to show that $\Isom(X)_\xi$ is amenable. We are thus reduced to the previous case.
\end{proof}

{A metric space is called \textbf{quasi-geodesic} if there is some constant $c$ such that any two points can be joined by a $c$-quasi-geodesic. For example, any orbit of a quasi-convex group of isometries in a hyperbolic metric space is quasi-geodesic. Remark that it is always possible to embed a quasi-geodesic subset coboundedly into a geodesic space (by gluing geodesic paths), but it is delicate to get a proper geodesic space.}
			
\begin{lem}[S. Adams]\label{lem:AmyStab}
Let $X$ be a proper quasi-geodesic hyperbolic space having a cocompact isometry group (or more generally, having bounded geometry), then for every $\xi \in \bd X$, the stabilizer $\Isom(X)_\xi$ is amenable. 
\end{lem}

\begin{proof}
See  \cite[6.8]{Adams96}; another simpler proof can be found in~\cite{kaimanovich2004boundary}. Remark that in \cite{Adams96}, the result is stated assuming $X$ is geodesic, but quasi-geodesic is enough with no alteration of the proof, and is a more robust statement because it allows passing to a quasi-convex subset. 
\end{proof}

\begin{remark}
If $X$ is proper, one can also use the amenability of $\Isom(X)_\xi$ to prove the existence of a (non-canonical) $\Isom(X)_\xi$-invariant function $X\times X\to\RR$ at bounded distance of the horokernels based at $\xi$. Indeed, the Tychonoff theorem implies that $\mathcal{H}_\xi$ is compact and the desired function is obtained by integrating an invariant measure on $\mathcal{H}_\xi$. More generally, one can use the amenability of the $\Isom(X)$-action on $\partial X$ to make an $\Isom(X)$-equivariant choice of such Busemann-like functions, depending measurably on $\partial X$.
\end{remark}

\begin{example}\label{exfocreg}
There are many instances where the Busemann quasicharacter is a character: the trivial case of bounded/horocyclic actions, and also the case when $X$ is proper, by Corollary~\ref{cor:Buse-char}. It is also the case when the hyperbolic space $X$ is CAT(0): indeed in this case there is a \emph{unique} horokernel based at each $\xi\in\partial X$. Therefore, this unique horokernel coincides with the Busemann function $b_\xi$ and moreover we then have $\beta_\xi(g) = b_\xi(x, gx)$ for all $g\in \Isom(X)_\xi$ and all $x\in X$.

On the other hand, here is an example of an oriented lineal action where the Busemann character $\beta_\xi$ is {\bf not} a homomorphism. Consider the centralizer in $\textnormal{Homeo}(\RR)$ of the translation $t\mapsto t+1$. This can be interpreted as the universal covering of the group of oriented homeomorphisms of the circle, and we thus denote it by $H=\widetilde{\textnormal{Homeo}}^+(\RR/\ZZ)$. Endow $\RR$ with the structure of Cayley graph with respect to the generating set $[-1,1]$, i.e.\ with the incidence relation $x\sim y$ if $|x-y|\le 1$. This incidence relation is preserved by the action of $H$, which thus acts on the Cayley graph. This Cayley graph is obviously quasi-isometric to $\RR$ and this action is transitive and lineal. If $\xi=+\infty$, then $\beta_\xi$ is the translation number $\beta_\xi(g)=\lim g^n(0)/n$, which is a classical example of non-homomorphic homogeneous quasicharacter. Actually, in restriction to some suitable subgroups (e.g.\ the inverse image of $\textnormal{PSL}_2(\ZZ)$ in $H$), this quasicharacter remains non-homomorphic and this provides examples where the acting group is finitely generated.

Here is now an example of a focal action, based on the same group. Set $C=\ZZ/2\ZZ$ (or any nontrivial finite group). Consider the permutational wreath product $G=C\wr_\RR H=C^{(\RR)}\rtimes H$, where 
$H$ acts by shifting the indices in $C^{(\RR)}=\bigoplus_{t\in\RR}C$ according to its action on $\RR$. 
Let $A\subset C^{(\RR)}$ be the subgroup of elements with support in $[0,+\infty)$ and let $W\subset H$ be the set of elements with translation number in $[-1,1]$ (i.e.\ those elements $\phi$ such that $\phi(0)\in [-1,1]$). Then $G$ is generated by $A\cup W$ and the corresponding Cayley graph is hyperbolic, the action of $G$ being focal. To see this, first observe that if $\alpha\in H$ be the translation $t\mapsto t+1$, then $\langle\alpha\rangle$ is cobounded in $H$ and thus it suffices to check that $C^{(\RR)}\rtimes\langle\alpha\rangle$, with the word metric associated to the generating set $W\cup\{\alpha\}$, is hyperbolic focal. But this is indeed the case   by Theorem~\ref{thm:Contracting}. 
\end{example}

\section{Focal actions and confining automorphisms}\label{sec:Focal}

Recall from the introduction that the action of a group $\Gamma$  on a hyperbolic metric space $X$ is called {\bf focal} if  it fixes a unique boundary point $\xi$ and if some element of $\Gamma$ acts as a hyperbolic isometry. Let  $\beta:\Gamma\to\RR$ the associated Busemann quasicharacter, as in \S\ref{hbq}.  The action of $\Gamma$ is said to be {\bf regular focal} if $\beta$ is a homomorphism. This holds in particular  if $X$ is CAT(0), or if  $X$ is proper (i.e.\ if balls are compact). 
The latter case will in fact crucial in the proof of Theorem~\ref{thm:main}. Example~\ref{exfocreg} illustrates that a focal action need not be regular in general.

\medskip
Let $H$ be a group (with no further structure a priori) and let $\alpha$ be an automorphism of $H$ and $A$ a subset of $H$. We say that the action of $\alpha$ is {\bf [strictly] confining $H$ into $A$} (we omit $H$ when no ambiguity incurs) if it satisfies the following three conditions
\begin{itemize}
\item $\alpha(A)$ is [strictly] contained in $A$;

\item  $H=\bigcup_{n\geq 0}\alpha^{-n}(A)$;

\item $\alpha^{n_0} (A\cdot A)\subset A$ for some non-negative integer $n_0$.
\end{itemize}

In case $H$ is a locally compact group, there is a close relation between confining and compacting automorphisms, which will be clarified in Corollary~\ref{cor:CompactingConfining} below. 

\medskip
Notice that the group $G=H \rtimes \la \alpha \ra$ is generated by the set $S=\{\alpha^{\pm}\}\cup A$. Endow $G$ with  the word metric $d_S$ associated to $S$. 
Given an action of a group $\Gamma$ on a metric space $X$ and a point $x\in X$, define a pseudo-metric on $\Gamma$ by $d_x(g,h)=d(g\cdot x,h\cdot x)$.

\begin{thm}\label{thm:Contracting}
Let $\Gamma$ be a group with a cobounded isometric action on a geodesic metric space $X$. Then the following assertions are equivalent.
\begin{enumerate}
\item\label{xhyp} $X$ is hyperbolic and the $\Gamma$-action is regular focal;

\item\label{cobo}There exist an element $\alpha\in\Gamma$ and a subset $A\subset [\Gamma,\Gamma]$ such that
\begin{itemize}
\item the image of $\alpha$ in $\Gamma/[\Gamma,\Gamma]$ has infinite order;
\item the action of $\alpha$ on $[\Gamma,\Gamma]$ is confining into $A$;
\item setting $G=[\Gamma,\Gamma]\rtimes\langle\alpha\rangle$ and $S=A\cup\{\alpha^\pm\}$, the inclusion map 
$(G,d_S)\to (\Gamma,d_x)$ is a quasi-isometry for some (hence every) $x\in X$.
\end{itemize}
\end{enumerate}
Moreover if (\ref{cobo}) holds, the Busemann character in (\ref{xhyp}) is proportional, in restriction to $G$, with the obvious projection to $\ZZ$.
\end{thm}

The implication (\ref{cobo})$\Rightarrow$(\ref{xhyp}) includes the fact that for every $G=H\rtimes\ZZ$ as above, $(G,d_S)$ is Gromov-hyperbolic (see Proposition~\ref{prop:Contracting}); this remains true when $\alpha(A)=A$ but in this case $(G,d_S)$ is elementary hyperbolic and quasi-isometric to the real line.

Beyond the locally compact case, a simple example of a group $H$ as above is a Banach space, $A$ being the unit ball and $\alpha$ being the multiplication by some positive scalar $\lambda<1$.

\subsection{From focal actions to focal hyperbolic groups}

 The following proposition reduces the proof of Theorem~\ref{thm:Contracting} to a statement in terms of metric groups:
A group $(\Gamma,d)$ is regular focal if and only if it has a subgroup   $G$ with a semidirect decomposition $G= H \rtimes \la \alpha \ra$ and a subset $A$ so that $\alpha$ is confining into $A$ and  the inclusion map 
$(G,d_S)\to (\Gamma,d)$ is a quasi-isometry.

\begin{prop}\label{prop:action/group}
Let $G$ be a group acting by isometries on a hyperbolic metric space $X$, and let $o\in X$. Let $d_G$ be any left-invariant pseudo-metric on $G$ such that the orbit map $(G,d_G)\rightarrow G\cdot o$ is a quasi-isometry. Then 
\begin{itemize}
\item[(i)] the action is focal if and only if 
$(G,d_G)$ is hyperbolic and the left $G$-action on $(G,d_G)$ is focal;  
\item[(ii)]  the action is regular focal if and only if 
$(G,d_G)$ is hyperbolic and the left $G$-action on $(G,d_G)$ is regular focal.
 \end{itemize}
\end{prop}

Let us start with two useful lemmas.

\begin{lem}
Let $f$ be a homogeneous quasicharacter on a group $G$. Suppose that $f$ is bounded in restriction to some normal subgroup $N$. Then $f$ induces a (homogeneous) quasicharacter on $G/N$. In particular if all homogeneous quasicharacter of $G/N$ are characters, then so is $f$.
\end{lem}
\begin{proof}
Let $g\in G$ and $n\in N$. Using that $f$ is bounded on $N$, we get $$f(gn)=f((gn)^k)/k=f(g^kn_k)/k= f(g^k)/k+A/k=f(g)+A/k$$ for some $n_k\in N$ and some bounded $A=A(g,n,k)$. Letting $k$ tend to infinity, we obtain that $f(gn)=f(g)$, which proves the lemma.
\end{proof}

\begin{lem}\label{lem:quasicharacter}
Let $f_1$ and $f_2$ be homogeneous quasicharacters on a group $G$ such that $$C^{-1}f_1-C\leq f_2\leq Cf_1+C$$ for some $C\geq 1$. If $f_1$ is a character, then so is  $f_2$.
\end{lem}
\begin{proof}
Since $f_1$ and $f_2$ are homogeneous, we have $C^{-1}f_1 \leq f_2\leq Cf_1$. Let $N = \Ker(f_1)$. By the previous lemma, both $f_1$ and $f_2$ induce a homogenous quasicharacter of the quotient $G/N$, and their value on $G$ is completely determined by their value on $G/N$. Since the quotient $G/N = f_1(G)$ is abelian, all its homogeneous quasicharacters are characters, and the claim follows.
\end{proof}

\begin{proof}[Proof of Proposition~\ref{prop:action/group}]
Recall that the orbits of a focal action are quasi-convex (Proposition~\ref{prop:undistorted}), hence are quasi-geodesic subspaces of $X$. Therefore, 
the fact that the action is focal (resp. regular focal) or not can be read on the restriction of the action on one orbit of $G$. In other words, this proves the proposition when $d_G$ is exactly the distance  induced by the orbit map. Now we need to prove that being focal (resp. regular focal) for a metric group $G$ only depends on a choice of metric {\it up to quasi-isometry}. This is clear for focal, since a quasi-isometry induces a homeomorphism between the boundaries, and therefore does not change the dynamics of the $G$-action on its boundary. The quasi-isometric invariance of the  regularity condition follows from Lemma~\ref{lem:quasicharacter}. 
\end{proof}

\subsection{From regular focal groups to confining automorphisms}

The implication (i)$\Rightarrow$(ii) in Theorem~\ref{thm:Contracting} will be deduced from the following.

\begin{prop}\label{prop:coco}
Let $(\Gamma,d)$ be a regular focal hyperbolic metric group. Let $\xi$ be the unique fixed point of the boundary and $\beta_\xi$ be the corresponding Buseman character. 
Set $H = [\Gamma, \Gamma] \subseteq \Ker \beta_\xi$ and let $\alpha \not \in \Ker \beta_\xi$. Then
\begin{itemize}
\item[(i)] $\la H \cup \{\alpha\} \ra \cong H\rtimes\langle\alpha\rangle$ is a 
cobounded, normal subgroup of $(\Gamma,d)$. (In particular, if $\Gamma$ is locally compact and the action continuous and proper, it is a 
cocompact normal subgroup of $\Gamma$.)
\item[(ii)] There exist $r_0>0$ satisfying:  for all $r>0$ there exists $n_0$ such that for all $n\geq n_0$, $\alpha^n(B(1,r)\cap H) \subset B(1,r_0)\cap H.$ In particular, $\alpha$ is confining into $A=B(1,r_0)\cap H$.
 \end{itemize}
\end{prop}
\begin{proof}
We start with a preliminary observation. Since $\alpha$ acts as a hyperbolic isometry, the focal point $\xi$ must be either its attracting or repelling fixed point. Upon replacing $\alpha $ by $\alpha\inv$, we may assume that it is the attracting one. In particular, the sequence $(1, \alpha, \alpha^2, \dots)$ defines a quasi-geodesic ray tending to $\xi$. Therefore, so is the sequence $(g, g\alpha, g\alpha^2, \dots)$ for any $g \in \Gamma$. Recall that here is a constant $r_0$ (depending only on the hyperbolicity constant of $(\Gamma, d)$) such that any two quasi-geodesic ray with the same endpoint are eventually $r_0$-close to one another. In particular, if $|\beta_\xi(g)| \leq C$, then $d(\alpha^n, g\alpha^n) \leq C + r_0$ for all $n$ larger than some $n_0$, where $n_0$ depends only on $d(1, g)$. 

\medskip
We now turn to the assertion~(i). The only nontrivial statement is that
$H\rtimes\langle\alpha\rangle$ is cobounded. 

Let thus $g \in \Gamma$ be arbitrary, and let $k \in \ZZ$ be such that $|\beta_\xi(g) - \beta_\xi(\alpha^k) | <  C =  |\beta_\xi(\alpha)|$. By the preliminary observation, we have $d(g \alpha^n, \alpha^{n+k}) \leq C+ r_0$. Therefore we have 
$$
d(g, [g, \alpha^{-n}] \alpha^k) = d(g, g\alpha^{-n} g\inv \alpha^{n+k}) \leq C+ r_0.
$$
Since $[g, \alpha^{-n}] \alpha^k \in H \rtimes \la \alpha \ra$, this proves that $H \rtimes \la \alpha \ra$ is $(C+r_0)$-dense in $\Gamma$, as desired.

\medskip
The assertion~(ii) also follows from  the preliminary observation, since for all  $g \in H$, we have $\beta_\xi(g) = 0$. 
\end{proof}

\subsection{From confining automorphisms to hyperbolic groups} 

We now turn to the converse implication in Theorem~\ref{thm:Contracting}, which is summarized in the following proposition. 

\begin{prop}\label{prop:Contracting}
Let $H$ be a group and let $\alpha$ be an automorphism of 
$H$ which confines $H$ into some subset $A \subset H$. 
 Let $S=\{\alpha^{\pm}\}\cup A$. 
Then  the group $G= H \rtimes \la \alpha \ra$ is Gromov-hyperbolic with respect to the left-invariant word metric associated to the generating set~$S$. If the inclusion $\alpha(A)\subset A$ is strict, then it is focal.
\end{prop}

Upon replacing $A$ by $A \cup A\inv \cup \{1\}$, we may assume that $A$ is symmetric and contains $1$. The group $H \rtimes \la \alpha \ra$ is endowed with the word metric associated with the symmetric generating set $S = \{\alpha^\pm \} \cup A$. Remark that this metric is \textbf{$1$-geodesic}, in the sense that for all $x,x'\in H \rtimes \la \alpha \ra$ at distance $\le n$ from one another, there exists a so-called \emph{$1$-geodesic} between them, i.e.\ $x=x_0,\ldots, x_n=x'$ such that $d(x_i,x_{i+1})=1$.  Denote by $B(n)=S^n$ the $n$-ball in this metric.

The following easy but crucial observation is a quantitative version of the fact that unbounded horocyclic actions are always distorted, see Proposition~\ref{prop:undistorted}. 

\begin{lem}\label{lem:Hdistorted}
There exists a positive integer $k_0$ such that all 1-geodesics of $G = H \rtimes \la \alpha \ra$ contained in $H$ have length $\leq k_0$.
\end{lem}
\begin{proof}
 Actually, we will prove a stronger statement, which, roughly speaking, says that $H$ is \emph{exponentially distorted} inside $H \rtimes \la \alpha \ra$.  Note that 
 $$A^2\subset  \alpha^{-n_0}(A).$$ 
 Since  $S = \{\pm \alpha  \} \cup A$,  we infer, more generally that
$$(B(1)\cap H)^{2^m}=A^{2^m}\subset \alpha^{-n_0m}(A)\subset B(2n_0m+1)\cap H.$$
Hence, if there exists a 1-geodesic of length $2^m$ contained in $H$, then $m$ must satisfy
$2^m\leq 2n_0m+1$, which obviously implies that it is bounded by some number $k_0$ depending only on $n_0$ (say, $k_0=4\log_2(n_0+2)$).  
\end{proof}

Now let us go further and describe 1-geodesics in $G = H \rtimes \la \alpha \ra$. Observe that a 1-geodesic between $1$ and $x$ can be seen as an element in the free semigroup over $S$ of minimal length representing $x$ (note that in this semigroup we do not have $ss^{-1}=1$; the reason we work in the free semigroup rather than free group is that the loop $s^ns^{-n}$ is not viewed as at bounded distance to the trivial loop). 

\begin{lem}\label{lem:geodesics}
{Every path emanating from $1$, of the form $\alpha^{n_1}h_1\alpha^{n_2}h_2\ldots \alpha^{n_k}h_k\alpha^{n_{k+1}}$  with $n_i \in \ZZ$ and $h_i \in H$, is at distance~$\leq k(k+1)$ from a path of the form $\alpha^{-i}g_1\ldots g_{k}\alpha^{j},$ with $i, j \geq 0$ and $g_i\in A$.

In particular, every $1$-geodesic from $1$ to $x$ in $H \rtimes \la \alpha \ra$ is at uniformly bounded distance from a word of the form $\alpha^{-i}g_1\ldots g_{k}\alpha^{j},$ with $i, j \geq 0$, $k\leq k_0$, and  $g_s\in A$ for all $s$, where $k_0$ is the constant from Lemma~\ref{lem:Hdistorted}.}
\end{lem}

\begin{proof}
{Let $m=\alpha^{n_1}h_1\alpha^{n_2}h_2\ldots \alpha^{n_k}h_k\alpha^{n_{k+1}}$ be a word in $S$ representing a path joining $1$ to some $x\in H$, such that $h_i\in A-\{1\}$ and $n_i\in\ZZ$. Note that every subword of the form $h \alpha^{-n}$ (resp. $\alpha^n h$), with $h\in A$, and $n\geq 0$  can be replaced by $\alpha^{-n} h'$ (resp. $h' \alpha^n $), with $h'=\alpha^n h\alpha^{-n} \in A$. Such an operation moves the $1$-geodesic to another $1$-geodesic at distance one (because $\alpha^{k}h\alpha^{-k} \in A$ for all $k=0,\dots,n$).

 With this process we can move positive powers of $\alpha$ all the way to the right, and negative powers to the left, obtaining after at most $k(k+1)$ operations a minimal writing of $x$ as $m'=\alpha^{-i}g_1\ldots g_{k}\alpha^j$, with  $i,j\geq 0$, and $g_s\in A$ for all $s$. This proves the first assertion. 
 
Assuming now that $m$ is of minimal length among words representing $x$.  In the above process of moving around powers of $\alpha$, the length of the word never gets longer, and since $m$ is minimal, it cannot get shorter either. Since the word $g_1\ldots g_{k}$ has minimal length, it forms a geodesic in $H$ and we deduce from Lemma~\ref{lem:geodesics} that $k\leq k_0$ and $m'$ is at distance at most $k_0(k_0+1)$ from $m$.}
\end{proof}

\begin{proof}[Proof of Proposition~\ref{prop:Contracting}]
Consider a 1-geodesic triangle $T$ in $H \rtimes \la \alpha \ra$. By Lemma~\ref{lem:geodesics}, we can suppose that $T$ is of the form 
$$T=[\alpha^{-i_1} u_1\alpha^{j_1}][\alpha^{-i_2} u_2\alpha^{j_2}][\alpha^{-i_3} u_3\alpha^{j_3}],$$ where $u_1,u_2,u_3$ are words of length $\le k_0$ in $A$ and $i_s, j_s \geq 0$. {Since $T$ forms a loop, its image under the projection map onto $\la \alpha \ra$ is also a loop, hence we have $i_1 + i_2 + i_3 =  j_1 + j _2 + j_3$. 

Let us prove that $T$ is thin, in the sense that every edge of the triangle lies in the $\delta$-neighborhood of the union of the two other edges.  Note that if after removing a backtrack in a triangle (in terms of words this means we simplify $ss^{-1}$), we obtain a $\delta$-thin triangle, then it means that the original triangle was $\delta$-thin. Upon removing the three possible backtracks, permuting cyclically the edges, and changing the orientation, we can suppose that
$$
T=[\alpha^{-k_1} u_1 \alpha^{k_2}] [u_2][\alpha^{-k_3} u_3],
$$
with $k_1, k_2 , k_3 \geq 0$ and $k_1 + k_3 = k_2$. 

By Lemma~\ref{lem:geodesics}, the $1$-path $u_2\alpha^{-k_3}$ is a distance $\le k_0$ from a $1$-geodesic segment of the form $\alpha^{-k_3}v_2$ (where $v_2$ has length $\le k_0$), and removing a backtrack by replacing $\alpha^{k_2}\alpha^{-k_3}$ by $\alpha^{k_1}$ we get a triangle $T'=[\alpha^{-k_1}u_1\alpha^{k_1}][v_2][u_3]$. Invoking Lemma~\ref{lem:geodesics} once more, we see that  $\alpha^{k_1}v_2u_3$ is a distance $\le 2k_0$ from a $1$-geodesic segment $w_2\alpha^{k_1}$, so $T'$ is at distance $\le 2k_0$ of the triangle $[\alpha^{k_1}][u_1][w_2\alpha^{k_1}]$, which after removing a backtrack is the bounded digon $[u_1][w_2]$, which is $k_0$-thin. Thus the original triangle is $4k_0$-thin and the space is $\delta$-hyperbolic with $\delta=16\log_2(n_0+2)$.}
\end{proof}

\subsection{Proof of Theorem~\ref{thm:Contracting}}

(i)$\Rightarrow$(ii) 
Let $d_o$ be the left-invariant pseudo-metric on $\Gamma$ defined by $d_o(g, h) = d(g.o, h.o)$, where $o \in X$ is a basepoint. 

By assumption, the $\Gamma$-action is regular focal and cobounded on the hyperbolic space $X$. 
By Proposition~\ref{prop:action/group}, the pair $(\Gamma, d_o)$ is a regular focal hyperbolic group. Let $\beta : \Gamma \to \RR$ be the associated Busemann character, let $H = [\Gamma, \Gamma] \subseteq \Ker \beta$ and $\alpha \not \in H$. By Proposition~\ref{prop:coco}, the group $G = H \rtimes \la \alpha \ra$ is a cobounded normal subgroup of $\Gamma$ (in particular $(G, d)$ is quasi-isometric to $(\Gamma, d_o)$, where $d$ denotes the restriction of $d_o$ to $G $), and $H$ contains some bounded subset $A$ such that $\alpha$ is confining into $A$. 

Let $S = \{\alpha^\pm\} \cup A$. It remains verify that the identity map  $(G,d_S)\to (G,d)$ is a quasi-isometry. First it is Lipschitz since $S$ is bounded for $d$. We then need to check that $(G,d)\to (G,d_S)$ is large-scale Lipschitz. Since $d$ is quasi-geodesic by Proposition~\ref{prop:undistorted}, it suffices to prove that a subset of $G$ is bounded for $d_S$ if it is bounded for $d$. In restriction to $H$, this 
follows from  Proposition~\ref{prop:coco} (ii). We conclude thanks to the fact that bounded sets of $(G,d)$ are mapped to bounded sets of $\RR$ by $\beta$.

\medskip
(ii)$\Rightarrow$(i) 
By Proposition~\ref{prop:Contracting}, the group $(G, d_S)$ is hyperbolic. Since $A^{2^n}\subset\alpha^{-nn_0}(A)$ for all $n$, it is clear that $A$ consists of non-hyperbolic isometries. If the inclusion $\alpha(A)\subset A$ is strict, then  the chain $\alpha^{-n}(A)$ is strictly ascending and thus $A$ is unbounded, hence its action on the Cayley graph is horocyclic and thus fixes a unique boundary point; in particular the action is not lineal. Since $\alpha$ acts as a hyperbolic element (because $\alpha^n$ has word length equal to $n$), the only possibility is that the action of $H\rtimes\langle\alpha\rangle$ is focal. (Obviously, if $\alpha(A)=A$, then $A=H$ is bounded and the action is lineal.) Since $(G, d_S)$ is quasi-isometric to $(\Gamma, d_x)$ which in turn is quasi-isometric to $X$ by hypothesis, the desired conclusion follows from Proposition~\ref{prop:action/group}. 
\qed

\section{Proper actions of locally compact groups on hyperbolic spaces}\label{sec:ProperActions}

In this section, we let  $G$ be a locally compact group acting isometrically continuously on the hyperbolic space $X$. We assume moreover that the action is \textbf{metrically proper}, i.e.\ the function $L(g)=d(x,gx)$ is proper for some (and hence all) $x\in X$.

\subsection{Preliminary lemmas}

\begin{lem}\label{lem:CptNorm}
If $G$ admits a non-horocyclic proper isometric action on a hyperbolic space (e.g.\ a cocompact action), then $G$ admits a maximal compact normal subgroup.
\end{lem}
\begin{proof}
Clearly, every compact normal subgroup is contained in the bounded radical $B_X(G)$. By Lemma~\ref{lem:bdrad}(\ref{bdr}), $B_X(G)$ acts with bounded orbits, so by properness, its closure is compact and thus $B_X(G)$ is compact.
\end{proof}

\begin{lem}\label{lem:AmenRad}
If $G$ admits a proper isometric action of general type on a hyperbolic space, then $G$ is non-amenable, and moreover the amenable radical of $G$ (the largest amenable normal subgroup of $G$) is compact.
\end{lem}
\begin{proof}
By Lemma~\ref{lem:schottky}, if $G$ admits an action of general type, then it contains a discrete nonabelian free subgroup and therefore is non-amenable.

Now let us prove the contrapositive statement. Let $G$ admit a proper action on a hyperbolic space and assume that its amenable radical $M$ is noncompact. Then the action of $M$ on $X$ is neither bounded nor general type, so is horocyclic, or lineal, or focal. In particular $M$ preserves a unique finite subset of cardinality at most~$2$ in $\bd X$.  This finite set is invariant under $G$ and it follows that the action of $G$ is not of general type.
\end{proof}

There is a partial converse to Lemma~\ref{lem:AmenRad}.

\begin{lem}\label{lem:amens}
Let $G$ admit a proper isometric action on a hyperbolic space. Assume that the action is bounded, or lineal, or focal. Then $G$ is amenable.
\end{lem}

\begin{proof}
By Proposition~\ref{prop:undistorted}, each $G$-orbit $Go$ is quasi-convex. Moreover the orbit  is closed and proper, by properness of the action. Thus each orbit is a proper quasi-geodesic hyperbolic space on which $G$ acts continuously, properly (and of course cocompactly). Therefore Lemma~\ref{lem:AmyStab} shows that $G$ is amenable.
\end{proof}

\begin{remark}The reader can find an example showing that the conclusion of Lemma~\ref{lem:CptNorm} can fail when the $G$-action is horocyclic.
Actually, the conclusion of Lemma~\ref{lem:amens} also fails in the horocyclic case. Indeed, as observed by Gromov \cite[\S 6.4]{Gro}, every countable discrete group has a proper horocyclic action on a hyperbolic space: it consists in endowing such a group with a left-invariant proper metric $d_0$ and define $d(g,h)=\log(1+d_0(g,h))$, and embedding it equivariantly as a horosphere into a hyperbolic space (which can be arranged to be proper). 
\end{remark}

\subsection{Proper actions of amenable groups}
 
Recall from \S\ref{hbq} that a Busemann quasicharacter  $\beta_\xi$ is defined on the stabilizer in $X$  of every boundary point $\xi \in \bd X$. 

\begin{prop}\label{prop:coctamen}
Let $X$ be a hyperbolic space.
Let $M$ be an amenable locally compact group with a continuous, (metrically) proper, cobounded isometric non-elementary action on $X$. Then we have the following.

\begin{enumerate}[(a)]
\item\label{it1} the action is focal, with a unique fixed boundary point $\xi \in \bd X$.

\item\label{it2} $\beta_\xi$ is a homomorphism and $\beta_\xi(M)$ is a closed non-zero subgroup of $\RR$. In particular, for every $\alpha\notin U = \Ker(\beta_\xi)$, the subgroup $\la U \cup \{\alpha \} \ra \cong U \rtimes \la \alpha \ra$ is closed, cocompact and normal in $M$.

\item\label{utrans} The action of $U$ on $\bd X \setminus \{\xi\}$ is proper and transitive.

\item\label{it5}  Every element $\alpha \in M$ with $ \beta_\xi(\alpha)<0$ acts as a compacting automorphism on $U$.
\end{enumerate}
\end{prop}

Note that by Lemma~\ref{lem:amens}, the conclusion of Proposition~\ref{prop:coctamen} applies to any proper focal isometric action of a locally compact group on an arbitrary hyperbolic space $X$ (in (\ref{utrans}), $\bd X$ has to be replaced by $\bd_X M$). 

\begin{proof}
(\ref{it1}) This follows from Lemma~\ref{lem:AmenRad}.

\medskip

(\ref{it2}) First, $\beta_\xi$ is a homomorphism by Corollary~\ref{cor:Buse-char}; since the action is focal,  it is nonzero by Lemma~\ref{lem:busatt}.
Suppose that $\beta_\xi(g_n)$ tends to $r\in\RR$. By Proposition~\ref{prop:coco}, we can write $g_n=h_ns_n$ with $d(s_nx,x)$ bounded and $\beta_\xi(h_n)\in\ZZ R$ for some $R>0$. By properness, we can extract, so that the sequence $(s_n)$ converges, say to $s\in G$. Observe that $\beta_\xi(h_n)=\beta_\xi(g_n) - \beta_\xi(s_n)$ converges by continuity of $\beta_\xi$, so the sequence $(\beta_\xi(h_n))$ eventually stabilizes. So $r=\beta_\xi(h_ns)$ for large $n$ and thus the image of $\beta_\xi$ is closed. 

\medskip 

(\ref{utrans}) Let us now verify that $U$ is transitive on $\bd X \setminus \{\xi\}$. The      
action being proper by Lemma~\ref{lem:properU}, we only need to check that      
the orbit of $\eta$ is dense. Now observe that the $U$-orbit of $\eta$ coincides with its $U \rtimes \la \alpha \ra$-orbit, which is dense by Lemma~\ref{lem:focaltt}.

\medskip

(\ref{it5}) Let $\eta$ be the second fixed point of $\alpha$. Let $W$ be the stabilizer of $\eta$ in $U$, so that $W$ is compact by properness (Lemma~\ref{lem:properU}). The homeomorphism $\varphi:U/W\to \bd X-\{\xi\}$ mapping $u\mapsto u\eta$, is equivariant when $U/W$ is endowed with the action induced by the conjugation action of the cyclic group $\langle\alpha\rangle$, namely $\varphi(\alpha u\alpha^{-1})=\alpha\varphi(u)$. Since $\alpha$ acts as a homeomorphism of $X-\{\xi\}$ contracting on $\{\eta\}$, it follows that the conjugation by $\alpha$ acts as a contracting homeomorphism of $U/W$ and therefore as a compacting automorphism of $U$.
\end{proof}

\subsection{Two-ended groups}

\begin{prop}\label{pro:debout}Let $G$ be a hyperbolic locally compact group. Then the visual boundary $\bd G$ consists of two points if and only if $G$ has a maximal compact normal subgroup $W$ so that $G/W$ is isomorphic to a cocompact closed group of isometries of the real line, and more precisely isomorphic to $\mathbf{Z}$, $\mathbf{Z}\rtimes\{\pm 1\}$, $\mathbf{R}$, or $\mathbf{R}\rtimes\{\pm 1\}$.
\end{prop}
It can be seen that these conditions also characterize two-ended groups among \emph{all} compactly generated locally compact groups. This thus improves a result of Houghton \cite[Theorem~3.7]{houghton1974ends}, who obtained a similar statement, but at the cost of passing to a closed cocompact normal subgroup.

\begin{proof}
 Write $\bd G=\{\xi,\eta\}$ and define $M$ as the stabilizer of $\xi$. By Lemma~\ref{lem:properU}, the action of $U=\Ker(\beta_\xi)$ on $\{\eta\}$ is proper and therefore $U$ is compact. By Proposition~\ref{prop:coctamen}(\ref{it2}), $D=M/U$ is isomorphic to either $\mathbf{Z}$ or $\mathbf{R}$. Note that $U$ is the kernel of the action on the boundary and therefore is normal in $G$. Set $E=G/U$.

Note that $D$ has index one or two in $E$. If $D=E$ we are done, so suppose that this index is two, i.e.\ the action of $E$ on its boundary is not trivial. Since $D$ is abelian and $E/D$ is cyclic; $D$ cannot be central in $E$, otherwise $E$ would be abelian and therefore isomorphic to $\ZZ$, $\RR\times\ZZ/2\ZZ$ or $\ZZ\times\ZZ/2\ZZ$ and these groups have a trivial action on their boundary. So for some $g\in E$, the action by conjugation on $D$ is given by multiplication by $-1$. Now $g^2$ commutes with $g$ but also belongs to $M$, so has to be fixed by multiplication by $-1$, so $g^2=1$, so that $E$ is isomorphic to $D\rtimes\{\pm 1\}$. This ends the proof.
\end{proof}

\subsection{Relative hyperbolicity}\label{sec:RelHyp}

In this paragraph, $X$ is a proper geodesic space. Let $G \leq \Isom(X)$ be a closed subgroup. A point $\xi \in \bd X$ is called a \textbf{conical limit point} (with respect to $G$) if for some (hence any) geodesic ray $\ro$ pointing to $\xi$, there  is some tubular neighbourhood of $\ro$ that intersects a $G$-orbit in an unbounded set. A point $\xi \in \bd X$ is called \textbf{bounded parabolic} if the stabilizer $G_\xi$ acts cocompactly on $\bd X \setminus \{\xi\}$. 
Following P.~Tukia~\cite{Tukia},  
we say that the $G$-action on $X$
is \textbf{cusp-uniform} if every point in $\bd X$ is a conical limit point or a non-isolated bounded parabolic point. In general, a locally compact group is said to be {\bf relatively hyperbolic} if it admits a proper cusp-uniform continuous action on some proper hyperbolic geodesic space. The stabilizers in $G$ of the boundary points that are not conical limit points are called \textbf{peripheral subgroups} and $G$ is more precisely called hyperbolic relatively to the collection of (conjugacy classes of) peripheral subgroups. We refer to~\cite{Yaman} for a proof of the equivalence between this definition and other characterizations of relative hyperbolicity in the discrete setting. 

\begin{lem}\label{lem:bordplein}
For any $G\leq\Isom(X)$, any point $\xi$ in $\bd X-\bd_X G$ is neither conical, nor non-isolated bounded parabolic.
\end{lem}

\begin{proof}
Obviously $\xi$ is not conical. Fix $x\in X$. If $\xi$ is non-isolated and $V$ is a small enough neighbourhood of $\xi$ then the convex hull of $V$ is disjoint from the orbit $Gx$, and therefore $\xi$ is not bounded parabolic.
\end{proof}

The following appears as Lemma~3.6 in \cite{kapovich1996greenberg}.
			
\begin{lem}\label{lem:full_limit_set}
If $G\le\Isom(X)$ is quasi-convex and $\bd_XG=\bd X$, then $G$ acts cocompactly.
\end{lem}
\begin{proof}
Otherwise, we can find in $X$ a sequence $(x_n)$ such that $d(x_n,Gx_0)$ tends to infinity. Translating $x_n$ by an element of $G$, we can suppose that $d(x_n,Gx_0)=d(x_n,x_0)$. From the quasi-convexity of $Gx_0$ it follows that for every sequence $(y_n)$ in $Gx_0$, the Gromov product $(x_n|y_n)=\frac12(d(x_n,x_0)+d(y_n,x_0)-d(x_n,y_n))$ does not tend to infinity and this means that any limit point of $(x_n)$ lies in $\bd X-\bd_XG$.
\end{proof}

Theorem~\ref{thm:RelHyp} from the introduction is an immediate consequence of the following.

\begin{prop}\label{prop:RelHyp:bis}
If $G\le\Isom(X)$ is  cusp-uniform, then
\begin{itemize}
\item either $G$ is cocompact (hence hyperbolic),
\item or the action is of general type (hence $G$ is not amenable). 
\end{itemize}
\end{prop}
\begin{proof}
First observe that the action is not horocyclic: indeed, for a horocyclic action with fixed point $\xi$, any point boundary point $\eta \neq \xi$ is neither bounded parabolic, nor conical.

Suppose that the action is not of general type and let us show that $G$ is cocompact. By Proposition~\ref{prop:undistorted}, $G$ is quasi-convex. By Lemma~\ref{lem:bordplein}, we have $\bd X=\bd_X G$ and by Lemma~\ref{lem:full_limit_set}, this shows that $G$ acts cocompactly. 
\end{proof}

\subsection{Groups of general type}

The following is a slight improvement of     
\cite[Th.~21]{MiMoSha} (in the cocompact case) and \cite[Theorem~1]{hamenstadt2009isometry} (the improvement being that in the rank one case we do not need to pass to a finite index subgroup). It will be used in the proof of Theorem~\ref{thm:application}. The proof we provide is substantially simpler than the previous ones.

\begin{prop}\label{pro:gt}
Let $X$ be a proper geodesic hyperbolic space and let $G\leq\Isom(X)$ be a closed subgroup of general type (i.e.\ non-amenable). 

Then $G$ has a unique maximal compact normal subgroup $W$, and $G/W$ is a virtually connected simple adjoint Lie group of rank one (actually, either the full group of isometries, or its identity component, which always has index at most two), or $G/W$ is totally disconnected.
\end{prop}

Geometrically, this means that a hyperbolic locally compact group of general type has a proper, isometric and cocompact action on either a rank one symmetric space of noncompact type, or a vertex-transitive locally finite graph. We first prove the following lemma.

\begin{lem}
Let $X$ be a proper geodesic hyperbolic space, let $G\leq\Isom(X)$ be a closed subgroup and  $W \lhd G$ a compact normal subgroup. 

Then for every non-elementary closed subgroup $H$ of $G$, the centralizer $Z=\{g\in G:\;[g,H]\subset W\}$ of $H$ in $G$ modulo $W$, is compact.\label{lem:centra}
\end{lem}
\begin{proof}
Note that $W$ acts trivially on the boundary. The assumption implies that $H$ contains two hyperbolic elements with distinct axes, so each of these axes have to be preserved by $Z$. So the action of $Z$ is bounded. Since $Z$ is clearly closed, we are done.
\end{proof}

\begin{proof}[Proof of Proposition~\ref{pro:gt}]
Let $V$ be the amenable radical of $G^\circ$. Since $G$ is of general type, $V$ is compact by Lemma~\ref{lem:AmenRad}. Moreover the quotient $S=G^\circ/V$ is semisimple with trivial center and no compact factor. A first application of Lemma~\ref{lem:centra} implies that either $S=1$ (in which case we are done) or $S$ is noncompact simple with trivial center, as we now suppose. Let $P/V$ be a maximal parabolic subgroup in $S$ (``parabolic" is here in the sense of algebraic/Lie groups). If $S$ has rank at least two, then $P$ is non-amenable, hence its action on $X$ is of general type, but $P$ has a noncompact amenable radical, contradiction. So $S$ has rank one. Let $\phi:G\to\textnormal{Aut}(S)$ be induced by the action by conjugation and $W$ its kernel. Then $V\subset W$ and $\phi$ induces a topological isomorphism from $S$ to $\textnormal{Aut}(S)^\circ$. 


 The kernel $W$ is compact: this follows from a second application of Lemma~\ref{lem:centra}, because $W$ is the centralizer modulo the compact normal subgroup $V$ of the non-amenable closed subgroup $G^\circ$. Thus $G/W$ is isomorphic to an open subgroup of $\textnormal{Aut}(S)$.

The precise statement follows from the knowledge of $\textnormal{Out}(S)$ (see \cite[Cor.~2.15]{gundoga} for a comprehensive list of automorphism groups of adjoint simple connected Lie groups): in the case of the real or complex hyperbolic space, $\textnormal{Out}(S)$ is cyclic of order two, while in the quaternionic or octonionic case, it is trivial. (Note that in complex hyperbolic spaces of even complex dimension, both components consist of orientation preserving isometries.) 
\end{proof}




\section{Structural results about compacting automorphisms}\label{sec:compaction}

\selectlanguage{french}
\begin{flushright}
\begin{minipage}[t]{0.7\linewidth}
\small\itshape  {\og Brigitte \'etait un de ces caract\`eres qui, sous le marteau de la pers\'ecution, se serrent, deviennent compactes.\fg}
\upshape
\begin{flushright}
(Balzac, Les Petits Bourgeois, Calmann L\'evy, 1892, p.\ 18.)
\end{flushright}
\end{minipage}
\end{flushright}
\selectlanguage{english}  

\fixbug

Before tackling the proof of our main theorem, we need to collect a few basic facts on the algebraic structure of 
groups admitting compacting automorphisms. This is the purpose of the present section. 

\medskip

Let $H$ be a locally compact group and $\alpha \in \Aut(H)$ be an automorphism. Recall from the introduction that $\alpha$ is said to be a  \textbf{compacting automorphism} (or a \textbf{compaction} for short) if there is some compact subset $V \subseteq H$ such that for all $g \in H$ there is some $n_0 \geq 0$ such that $\alpha^n(g) \in V$ for all $n > n_0$. Such a subset $V$ is called a {\bf pointwise vacuum set} for $\alpha$.


\subsection{Basic features of compactions and proof of Theorem~\ref{thm:Amen:intro}}\label{sec:ProofThmA}


For the sake of future reference, we collect some basic properties of compactions, which will be frequently used in the sequel. 

Following~\cite{CoTe} we define a {\bf vacuum set} as a subset $\Omega$ such that for every compact subset $M$ of $H$, there is some $n_0 \geq 0$ such that $\alpha^n(M)\subset \Omega$ for all $n > n_0$. A vacuum set is obviously a pointwise vacuum set. The following proposition gives a partial converse, showing that compactions are uniform on compact subsets. In the case of contractions, the corresponding result is due to S.P.~Wang \cite[Prop.~2.1]{wang1984mautner}.

\begin{prop}\label{wangbis}
Let $H$ be a locally compact group and $\alpha$ an automorphism.
Then $\alpha$ is a compaction if and only if there is a compact vacuum set for $\alpha$. 
\end{prop}

In other words, compactions are ``uniform on compact subsets", and thus our choice of terminology is consistent with that in \cite{CoTe} (modulo the fact that compactions are termed ``contractions" there).

\begin{proof}[Proof of Proposition~\ref{wangbis}]
Assume that $C$ is a compact pointwise vacuum subset, symmetric and containing~1. We are going to show that for some compact normal subgroup $W$, the compact subset $W\cdot C\cdot C\cdot C\cdot C$ is a vacuum set. 

Since $H=\bigcup_{n\ge 0}\alpha^{-n}(C)$, it is $\sigma$-compact. By~\cite{KakutaniKodaira}, $H\rtimes_\alpha\mathbf{Z}$ has a compact normal subgroup $W$ such that $(H\rtimes\ZZ)/W$ is second countable. We may thus assume henceforth that $H$ is second countable. 

For each subset $A \subseteq H$ and all $m\ge 0$,  we set  $A_m= \bigcap_{n \geq m} \alpha^{-n}(A)$. 
Clearly we have $A_0 \subseteq A_1 \subseteq \dots$ and $A_m \cdot B_m \subseteq (A\cdot B)_m$ for all $A, B \subseteq H$. 

Let now $C \subseteq  H$ be a compact subset such that  $\alpha^n(g) \in C$ for all $g\in H$ and all sufficiently large $n$. Upon replacing $C$ by $C \cup C\inv \cup \{1\}$, we may assume that $ C = C\inv$ and contains the identity.
Notice that we have $H = \bigcup_{m \geq 0} C_m$. By Baire's theorem (which can be invoked since $H$ second countable), we deduce that $C_m$ has non-empty interior for some $m>0$. Since $C_m \cdot C_m\inv \subseteq (C \cdot C\inv)_m = (C\cdot C)_m$, it follows that the compact set $A = C\cdot C$ has the property that $A_m$  contains some identity neighbourhood and $\bigcup_{n \geq 0} A_n = H$. 

Let now $\Omega \subseteq H$ be any compact subset. Then there is a finite set $g_1, \dots, g_k \in \Omega$ such that $\Omega \subseteq \bigcup_{i=1}^k g_i A_m$. Since $A_m \cup \{g_1, \dots, g_k\} \subseteq A_n$ for all sufficiently large $n>0$, we deduce that the compact set  $B = A \cdot A$ contains $\alpha^n(\Omega)$ for all sufficiently large $n$. 
\end{proof}

We record the following relation between \emph{compacting} and \emph{confining} automorphisms, as defined in Section~\ref{sec:Focal} above. 

\begin{cor}\label{cor:CompactingConfining}
Let $H$ be a  [noncompact] locally compact group and $\alpha \in \Aut(H)$. 

Then $\alpha $ is compacting if and only if  there is a compact subset $A$ into which $\alpha$ is [strictly] confining. 
\end{cor}

\begin{proof}
Assume that $\alpha$ is compacting. By Proposition~\ref{wangbis}, there is a compact vacuum set $\Omega$ for $\alpha$. Then $\alpha^{m}(\Omega) \subseteq \Omega$ for some $m > 0$. Set 
$$A = \Omega \cup \alpha(\Omega)  \cup \dots \cup \alpha^{m-1}(\Omega).$$ 
Thus $A$ is itself a compact vacuum set; moreover it satisfies $\alpha(A) \subset A$. It follows that $\alpha$ is confining into $A$, as desired. 

The converse is obvious from the definitions. 
\end{proof}


\begin{proof}[Proof of Theorem~\ref{thm:Amen:intro}]
We saw in Corollary~\ref{cor:CompactingConfining}, as an application of Baire's theorem,  that if $\alpha$ is a compacting automorphism of a [noncompact] locally compact group $H$, then $H$ admits a compact subset $A$ into which $\alpha$ is [strictly] confining. In particular, Theorem~\ref{thm:Contracting} (or Proposition~\ref{prop:Contracting}) implies that the semidirect product  $H \rtimes  \la \alpha \ra$ of a locally compact group $H$ with the cyclic group generated by a compacting automorphism, is hyperbolic. Since a group of the form  $H \rtimes  \RR$ contains a cocompact subgroup of the form  $H \rtimes  \ZZ$, the `if' part of Theorem~\ref{thm:Amen:intro} follows. 

Conversely, assume that $G$ is non-elementary hyperbolic and amenable. Recall from Corollary~\ref{cor:eqprop} that $G$ admits a continuous, proper cocompact action by isometries on a proper hyperbolic geodesic metric space. The desired conclusion then follows from Proposition~\ref{prop:coctamen}.
\end{proof}

We define the {\bf limit group} of a compacting automorphism as the intersection of all compact vacuum subsets. The definition makes sense in view of Proposition~\ref{wangbis}. The limit group need not be a vacuum set unless the ambient group is compact. The main properties of the limit group are collected in the following.

\begin{lem}\label{lem:LimitGroup}
Let $H$ be a locally compact group and $\alpha \in \Aut(H)$ be a compacting automorphism with limit group $L$. Then the following properties hold.
\begin{enumerate}[(a)]
\item\label{gc1} The limit group $L$ is an $\alpha$-invariant compact subgroup. 

\item\label{gc2} For every compact subset $K\subset H$ we have $\bigcap_{n\in\ZZ}\alpha^nK\subseteq L$, with equality if $K$ is a vacuum subset.

\item\label{gc3} The limit group $L$ is the largest $\alpha$-invariant compact subgroup.

\item\label{gc4}\label{acsivs} A compact subset $K$ of $H$ is a vacuum set if and only if it is a neighbourhood of $L$.
\end{enumerate}
\end{lem}

\begin{proof}
For any $K \subseteq H$, we set $L'(K)=\bigcap_{n \in \ZZ} \alpha^n K$. 

We start with a preliminary observation. If $\Omega$ is a vacuum set and $K \subset H$ is any compact subset, then $\alpha^n K \subset \Omega$ for all $n>0 $ sufficiently large. In particular, for any infinite set $I$ of positive integers, we have $\bigcap_{n \in I} \alpha^n K \subset \Omega$. This being valid for any vacuum set $\Omega$, we infer that $\bigcap_{n \in I} \alpha^n K \subset L$.  

\medskip \noindent (\ref{gc2})
The preliminary observation implies that $L'(K) \subseteq L$ for each compact subset $K$. If in addition $K$ is a vacuum set, then $\alpha^n K$ is a vacuum set for all $n$ so that $L \subseteq L'(K)$, whence $L = L'(K)$.

\medskip \noindent (\ref{gc1})
The limit group is obviously compact and $\alpha$-invariant (because $\alpha$ permutes the vacuum subsets), but we have to show that it is indeed a subgroup. Fix a vacuum set $\Omega$. For some $n_0$, we have $\alpha^{n_0}(\Omega\cdot\Omega)\subset\Omega$. In particular, if $x,y\in L'(\Omega)=\bigcap_n\alpha^n\Omega$, for all $n$ we have $\alpha^{-n}(x),\alpha^{-n}(y)\in\Omega$, so $\alpha^{-n}(xy)=\alpha^{-n}(x)\alpha^{-n}(y)\in\alpha^{-n_0}\Omega$, i.e.\ $xy\in\alpha^{n-n_0}(\Omega)$ for all $n$. Similarly one shows that for each $x \in L'(\Omega)$ we have $x\inv \in L'(\Omega)$. Thus $L'(\Omega)$ is a subgroup and (\ref{gc1}) follows from (\ref{gc2}). 


\medskip \noindent (\ref{gc3})
If $M$ is an $\alpha$-invariant compact subgroup, it is immediate that $M$ is contained in every compact vacuum subset, and therefore $M\subset L$.

\medskip \noindent (\ref{gc4})
Let $\Omega$ be a compact vacuum subset of $H$. Upon enlarging $\Omega$ if necessary, we may assume that $\Omega$ is   a neighbourhood of $L$. If $V$ is any vacuum subset, then for some $n>0$ we have $\alpha^n \Omega \subset V$. Since $L$ is $\alpha$-invariant, it follows that $\alpha^n \Omega$, and hence also $V$, is a neighbourhood of $L$. 

Let $V$ be  a neighbourhood of $L$ and suppose for a contradiction that $V$ is not a vacuum set. Then there exists a compact subset $K$, an infinite set $I$ of positive integers and a sequence $(x_n \in K)_{n \in I}$ such that $\alpha^n x_n \not \in V$ for all $n \in I$. Since $ \alpha^n K \subset \Omega$ for all sufficiently large $n>0$, we may assume after passing to a subsequence that $\alpha^n x_n$ converges to some $x$. Since $V$ is a neighbourhood of $L$, we  have $x \not \in L$. On the other hand, since for each vacuum set $\Omega'$, there is $n_0 > 0$ such that $\bigcup_{n > n_0} \alpha^n K \subseteq \Omega'$, it follows that $x$ belongs to all vacuum sets, and hence to $L$, which is absurd.
\end{proof}

Here are some more straightforward properties of compactions.

\begin{lem}\label{lem:cbp}
Let $H$ be a locally compact group and $\alpha \in \Aut(H)$. Then the following properties hold.
\begin{enumerate}[(a)]

\item If $K < H$ is a compact normal  subgroup invariant under $\alpha$, then $\alpha $ is a compaction if and only if it induces a compaction on $H/K$.  

\item\label{maxcpco} if $\alpha$ is compacting, the group $H\rtimes\langle\alpha\rangle$ has a maximal compact subgroup, which is the intersection of all $H$-conjugates of the limit group $L\subset H$.

\item\label{power}  $\alpha$ is a compaction if and only $\alpha^n$ is a compaction for some $n>0$. 

\item\label{continuous} If $(\alpha_t)$ is a continuous one-parameter subgroup of automorphisms of $H$, then if some $\alpha_t$ is a compaction for some $t>0$ then it is true for all $t>0$, and this implies that $H$ is connected-by-compact. 

\item\label{vacg} If $H$ is  totally disconnected, then $\alpha$ is a compaction if and only if some compact open subgroup of $H$ is a vacuum set for $\alpha$. 
\end{enumerate}
\end{lem}
\begin{proof}
For (\ref{continuous}), observe that $\RR$ acts as the identity on $H/H^\circ$ so the latter has to be compact if $\alpha_1$ is compacting. Then observe that $\bigcup_{t\ge 0}\alpha_t(\Omega)$ is a compact vacuum subset, and use (\ref{power}).

For (\ref{vacg}), use Lemma~\ref{lem:LimitGroup}(\ref{acsivs}) and the fact that in a totally disconnected group, every compact subgroup is contained in a compact open subgroup. 

Other verifications and details are left to the reader. 
\end{proof}

Recall that a locally compact group has {\bf polynomial growth} if for every compact subset $S$, the Haar measure of $S^n$ is polynomially bounded with respect to $n$.

\begin{prop}\label{compamenable}
If a locally compact group $H$ admits a compaction $\alpha$, then $H$ has polynomial growth (of uniform degree) and thus is unimodular and amenable. In particular, $G=H\rtimes_\alpha\ZZ$ is amenable, and is non-unimodular unless $H$ is compact.
\end{prop}
\begin{proof}
Let $V$ be a compact subset and $\lambda$ the Haar measure, and let us show that $\lambda(V^n)$ grows polynomially (with degree depending only on $H$ and $\alpha$). Replacing $V$ by a larger compact subset and $\alpha$ by a positive power if necessary, we can suppose that $\alpha(VV)\subset V$. Since $\alpha$ is a compacting automorphism, it multiplies the Haar measure of $H$ by some $1/c\le 1$. The above inclusion implies that for all $k\ge 0$ we have $V^{2^k}\subset \alpha^{-k}(V)$. So
$$V^n\subset V^{2^{\lceil \log_2 n\rceil}}\subset \alpha^{-\lceil \log_2 n\rceil}(V)$$
and thus
$$\lambda(V^n)\le \lambda(V)c^{\lceil \log_2 n\rceil}\le \lambda(V)cn^{\log_2c},$$
which is a polynomial upper bound of degree $\log_2c$. So $H$ has polynomial growth and therefore is unimodular and amenable. It follows that $G$ is amenable. 
If moreover $H$ is noncompact, it follows that $G$ is a semidirect product of two unimodular groups, where the action of the acting group does not preserve the volume, so $G$ is non-unimodular.
\end{proof}

\subsection{Reduction to connected and totally disconnected cases}

Given a locally compact group $G$, a subgroup $H \leq G$ is called \textbf{locally elliptic} if every finite subset of $H$ is contained in a compact subgroup of $G$. If $G$ is locally elliptic and second countable, then $G$ is a union of a countable ascending chain of compact open subgroups. By~\cite{Platonov}, the closure of a locally elliptic subgroup is locally elliptic, and an extension of a locally elliptic group by a locally elliptic group is itself locally elliptic. In particular any locally compact group $G$ has a unique maximal closed normal subgroup that is locally elliptic, called the \textbf{locally elliptic radical} of $G$ and denoted by $\LF(G)$. It is a characteristic subgroup of $G$, and the quotient $G/\LF(G)$ has trivial locally elliptic radical. (Some authors use the confusing terminology \emph{(topologically) locally finite} and {\em LF-radical} for locally elliptic groups and locally elliptic radical.) 

The following result was  established in~\cite{CoTe} (using the existence of a compact vacuum set as the definition of a compacting automorphism). 
 
\begin{prop}\label{prop:stcom}
Let $H$ be a locally compact group whose identity component $H^\circ$ has no nontrivial compact normal subgroup. If $H$ admits a compacting automorphism, then the product map $H^\circ \times \LF(H) \to H$ is an isomorphism onto an open finite index subgroup of $H$. 
\end{prop}

\begin{proof}
See Theorem~A.5 and Corollary A.6 from~\cite{CoTe}. 
\end{proof}

\subsection{Compacting automorphisms of totally disconnected groups}

The following lemma was established in Proposition~4.6 from~\cite{CoTe}.

\begin{lem}\label{lem:TD_compactable}
Let $H$ be a totally disconnected locally compact group and $\alpha$ be a compacting automorphism. Then $H$ is locally elliptic and $G=H \rtimes \la \alpha \ra$ acts properly and vertex and edge-transitively without inversions on a $(k+1)$-regular tree $T$, and fixes an end $\xi \in \bd T$. In particular $G$ is hyperbolic. The integer $k$ is characterized by the equality $\Delta(G)=\{k^n:\;n\in\ZZ\}$, where $\Delta$ is the modular function.  
\end{lem}
\begin{proof}[About the proof]
By Lemma~\ref{lem:cbp}(\ref{vacg}), there is a compact open subgroup $\Omega$ which is vacuum for the compaction; replacing $\Omega$ by $\bigcup_{n\ge 0}\alpha^n(\Omega)$ if necessary, we can suppose that $\alpha(\Omega)\subset\Omega$ and the associated tree is nothing else than the Bass--Serre tree associated to the corresponding ascending HNN-extension. 
\end{proof}

\subsection{Compacting automorphisms of almost connected groups}

Let $H$ be a virtually connected Lie group. Its {\bf nilpotent radical} is defined here as the largest normal connected nilpotent subgroup. It is closed and characteristic on $H$.

There is a general decomposition result for one-parameter groups of compactions.

\begin{prop}[Hazod-Siebert~\cite{HazodSiebert}]\label{hazsie}
Let $H$ be a connected-by-compact locally compact group with a continuous compacting action $\alpha$ of $\RR$. Then the nilpotent radical $N$ of $H^\circ$ is simply connected and there is an $\alpha$-invariant semidirect product decomposition $$H=N\rtimes K_\alpha,$$ where $K_\alpha$ is the limit group of $\alpha$; moreover $N$ exactly consists of the contracted elements (i.e.\ those $h\in H$ such that $\lim_{t\to +\infty}\alpha(t)(h)=1$).
\end{prop}

A similar result for arbitrary compactions holds under further assumptions.

\begin{prop}\label{prop:cpli}
Let $H$ be a virtually connected Lie group and $N$ its nilpotent radical. If $H$ has no nontrivial compact normal subgroup and admits a compacting automorphism $\alpha$, then $N$ is simply connected and 
for every compacting automorphism $\alpha$ of $H$, denoting by $K_\alpha$ its limit subgroup, we have $H=N\rtimes K_\alpha$ and $N$ consists exactly of the elements of $H$ that are contracted by $\alpha$.
\end{prop}

\begin{lem}\label{finimodn}
Let $G$ be a connected solvable Lie group, without nontrivial compact normal subgroups, and $N$ its nilpotent radical. Then for any automorphism $\alpha$ of $G$, the automorphism induced on $G/N$ has finite order.
\end{lem}
\begin{proof}
Consider the Zariski closure $P$ of $\langle\alpha\rangle$ in $\textnormal{Aut}(\mathfrak{g})$, and $P^\circ$ its identity component in the real topology. Write $G=\tilde{G}/Z$ with $Z$ discrete and central, and $\tilde{G}$ simply connected. Then $\tilde{G}\rtimes P^\circ$ is a connected solvable Lie group and therefore its derived subgroup is nilpotent. Now if $\alpha$ has infinite order as an automorphism of $G/N$, then some power of $\alpha$ belonging to $P^{\circ}$ acts nontrivially on $G/N$, so the connected group $[\tilde{G},P^{\circ}]$ strictly contains the inverse image $\tilde{N}$ of $N$ in $\tilde{G}$. Its image in $G$ is a connected nilpotent normal subgroup of $G$ strictly containing $N$, contradicting the maximality of $N$.
\end{proof}

\begin{lem}\label{gnk}
Let $G$ be a virtually connected Lie group without nontrivial compact normal subgroups, and $N$ its nilpotent radical, and assume that $G/N$ is compact. Then for any maximal compact subgroup $K$ of $G$, we have $G=N\rtimes K$.
\end{lem}
\begin{proof}
Since $G$ has no nontrivial compact subgroup, $N$ is simply connected and therefore $N\cap K=1$, and we also deduce that $G\to G/N$ is a homotopy equivalence. Since the inclusion $K\to G$ is also a homotopy equivalence \cite[Theorem~3.1]{mostow1955self}, the composite homomorphism $K\to G/N$ is injective and is a homotopy equivalence between (possibly not connected) compact manifolds. This is necessarily a homeomorphism and therefore $G=N\rtimes K$.
\end{proof}

\begin{proof}[Proof of Proposition~\ref{prop:cpli}]
Fix a compaction $\alpha$. By Proposition~\ref{compamenable}, $H$ is amenable. Let $R$ be the solvable radical of $H$ (the largest connected solvable normal subgroup), so $N\subset R$ and $H/R$ is compact by amenability. By Lemma~\ref{finimodn}, $\alpha$ induces an automorphism of finite order of $R/N$ and therefore, since $\alpha$ is compacting, we deduce that $R/N$ and hence $H/N$ is compact. 

Let $L$ be the limit group of $\alpha$. Then $\alpha$ induces a contraction of the manifold $H/L$; since $H/L$ is locally contractible, it immediately follows that $H/L$ is contractible (because it has all its homotopy groups trivial). Therefore $L$ is a maximal compact subgroup (indeed, if $L'$ is a maximal compact subgroup containing $L$, then $H/L$ retracts by deformation to $L'/L$, which is a compact manifold, so cannot be contractible unless it is reduced to a point). By Lemma~\ref{gnk}, we deduce that $H=N\rtimes L$.

Finally, the restriction $\alpha|_N$ is a compaction, but $N$ has no nontrivial compact subgroup, so the limit group of $\alpha|_N$ is trivial and hence $\alpha|_N$ is a contraction.
\end{proof}

A compaction as in Proposition~\ref{prop:cpli} cannot in general be extended to a one parameter subgroup; however we point out the following fact in this direction.

\begin{lem}\label{lem:CoctEmbedding}
Let $H$ be a virtually connected Lie group without nontrivial compact normal subgroup and $\alpha \in \Aut(H)$ be a compaction. 

Then there is a proper injective homomorphism with cocompact image
$$i:\;G= H\rtimes_\alpha \ZZ  \to G'=H'\rtimes \RR$$
with compacting action of $\RR$ on $H'$, so that $i(H)$ is open of finite index in $H'$, and $i$ induces the standard embedding $\ZZ\to\RR$ (on the compacting factors). Moreover $i(G)$ is normal in $G'$ and $G'/i(G)$ is abelian.
\end{lem}

To check the lemma, we use the following easy subsidiary fact.

\begin{lem}\label{paramab}
Let $G$ be the group of real points of an affine algebraic group defined over the reals, viewed as a locally compact group. Let $Z$ be a discrete subgroup of $G$ isomorphic to $\ZZ$. 

Then there exists a closed subgroup $P$ of $G$, contained in the Zariski closure of $Z$, isomorphic to $\RR\times F$ with $F$ finite cyclic, such that $Z$ is contained and cocompact in $P$. 
\end{lem}
\begin{proof}
Let $M$ be the Zariski closure of $Z$. As a compactly generated abelian Lie group with finitely many components, it is isomorphic to $M^\circ\times F$ with $M^\circ$ an abelian connected Lie group and $F$ finite \cite[II.~\S 2.1]{bourbaki2007theories}; we can replace $F$ by the projection of $Z$ so as to suppose it cyclic. The projection of $Z$ in $M^\circ$ is contained in a closed one-parameter subgroup; its inverse image in $M$ is the desired subgroup $P$.
\end{proof}

\begin{proof}[Proof of Lemma~\ref{lem:CoctEmbedding}]
We may assume $H$ non-compact, since otherwise $H$  would be trivial by hypothesis, in which case the desired conclusions are clear.  
Let $N$ be the nilpotent radical of $H$ and $K$ be the limit group of $\alpha$, so that $H = N \rtimes K$ by Proposition~\ref{prop:cpli}. Observe that $\Aut(N)$ is an affine algebraic group; let $\rho \colon  K\rtimes_\alpha \ZZ \to\Aut(N)$ be the homomorphism defined by the action by conjugation. The hypothesis that $\alpha$ is a compaction, together with the fact that $H$ is non-compact, implies that $\rho$ is injective and moreover that $\rho(\ZZ)$ is a discrete subgroup of $\Aut(N)$. Let  then $P=\RR\times F\supset\rho(\ZZ)$ be a subgroup of $\Aut(N)$  as given by Lemma~\ref{paramab}. Since $K$ is compact, so is $\rho(K)$, which is thus Zariski closed in $\Aut(N)$ by the Weierstrass approximation theorem.
Since $P$ is contained in the Zariski closure of $Z$, it normalizes $\rho(K)$. Therefore
$\rho(K)P$ is a virtually connected Lie group. The homomorphism $\rho \colon K\rtimes_\alpha \ZZ\to \rho(K)P$ is proper and has a cocompact image, containing the group of commutators  of $\rho(K)P$ since $P$ is abelian and $\rho(K)$ is normal. Thus we obtain a proper embedding 
$$N\rtimes(K\rtimes_\alpha \ZZ)\to G'=N\rtimes \rho(K)P$$ 
whose image contains $[G',G']$. Since $P = \RR \times F$, we may write $\rho(K)P=\rho(K)F\rtimes\RR$. We finally obtain the desired claim by setting $H'=N\rtimes \rho(K)F$.
\end{proof}

We deduce the following corollary, which will be used in the proof of Theorem~\ref{thm:main}.

\begin{cor}\label{Heintzeco}
Let $G=H\rtimes\RR$ or $H\rtimes\ZZ$ be a semidirect product of a locally compact group $H$ with compacting action of $\RR$ or $\ZZ$, and assume in the second case that $H$ is connected-by-compact. Then $G$ has a proper cocompact isometric action on a homogeneous negatively curved Riemannian manifold $X$ fixing at point at infinity. More precisely, in the case of $H\rtimes\RR$ this action is transitive, and in the case of $H\rtimes\ZZ$, this action has, as orbits, the subsets $\{x:\;b(x)-r\in\ZZ\}$ where $b$ is (a scalar multiple of) some Busemann function and $r\in\RR/\ZZ$; in particular, the orbit space $G\backslash X$ is a circle in the latter case. 
\end{cor}

\begin{proof}
Note that $H$ is connected-by-compact (this follows by Lemma~\ref{lem:cbp}(\ref{continuous}) in the first case, and by assumption in the second). Start by modding out by the maximal compact normal subgroup, so that in particular $H$ is a virtually connected Lie group. Second, in view of Lemma~\ref{lem:CoctEmbedding}, it suffices to deal with  the case of a semidirect product by $\RR$. 

So we assume that $G=H\rtimes\RR$ with compacting action. By Proposition~\ref{hazsie}, we can write $G=N\rtimes (K\rtimes\RR)$ with $K$ compact. By \cite[Proposition~4.3]{CoTe} (which relies on the results of Heintze~\cite{Heintze}), the connected manifold $G/K$ admits a $G$-invariant metric of negative curvature and the corollary is proved, the orbit description being clear.
\end{proof}

The following proposition will also be used in the proof of Theorem~\ref{thm:main}.

\begin{prop}\label{someall}
Let $H$ be a connected-by-compact locally compact group, and let $\Lambda$ be either $\RR$ or $\ZZ$, and denote by $\Lambda_+$ the set of positive elements in $\Lambda$. Let $G = H\rtimes\Lambda$ be a semidirect product such that some element $\alpha \in H\rtimes\Lambda_+$ compacts $H$. 
Then every element in $H\rtimes\Lambda_+$ compacts $H$.
\end{prop}

We need the following lemma.

\begin{lem}\label{inner}
Let $H$ be a locally compact group, $\alpha$ a compaction of $H$ and $\phi$ an inner automorphism of $H$. 
Then $\alpha\phi$ is a compaction as well.
\end{lem}
\begin{proof}
By the solution to Hilbert's fifth's problem, $H^\circ$ admits a maximal compact normal subgroup $W$; we freely use the observation that an automorphism of $H$ is compacting if and only it induces a compacting automorphism of $H/W$.

Write $c_x(t)=xtx^{-1}$. We have to 
prove that $\alpha\circ c_x$ is a compaction for every $x\in H$.
We have $$(\alpha c_x)^n(t)=c_{x(n)}(\alpha^n(t)),\text{ where }x(n)=\alpha(x)\dots \alpha^n(x).$$ So it is enough to check that the sequence $(x(n))$ is bounded (for each fixed $x$).

We begin with two particular easier cases. The first is when $x$ belongs to a closed, $\alpha$-invariant locally elliptic subgroup; then all $\alpha^n(x)$ belong to a single compact subgroup and thus $x(n)$ is bounded. The second is when $x$ belongs to a closed $\alpha$-invariant subgroup which is a simply connected nilpotent Lie group. Then $(\alpha^n(x))$ converges exponentially to zero and thus is summable, so $(x(n))$ is bounded.

In general, we claim that if $W$ is the maximal compact normal subgroup in $H^\circ$, then for every compaction $\alpha$ of $H$, the group $H/W$ is generated $E_1\cup N_1$, where $E_1$ is an $\alpha$-invariant closed locally elliptic subgroup and $N_1$ is a characteristic subgroup, which is a simply connected nilpotent Lie group. Granting the claim, starting from a compaction $\alpha$, 
and choosing $E_1$ and $N_1$ accordingly, we deduce from the first particular case that $\beta=\alpha\circ c_x$ is a compaction for all $x\in E_1$. Since $N_1$ is $\beta$-invariant, it follows from the second particular case that $\alpha\circ c_{xy}=\beta\circ c_y$ is a compaction for all $y\in N_1$. Since $N_1$ is normal, every element in $H$ can be written such a form $xy$ for $x\in E_1$ and $y\in N_1$, which completes the proof modulo the claim.

It remains to prove the claim. Modding out if necessary, we can suppose that $H^\circ$ has no non-trivial compact normal subgroup. Let $E$ be the locally elliptic radical and $C$ the identity component of $H$; by Proposition~\ref{prop:stcom}, $E$ and $C$ generate their topological direct product, and $E\times C$ stands an open subgroup of finite index of $H$. Let $W/E$ be the maximal compact normal subgroup of $H/E$. Then $W$ is locally elliptic, and it follows from the definition of $E$ that $W=E$. So we can apply Proposition~\ref{prop:cpli} and get an $\alpha$-invariant decomposition $H/E=N\rtimes K$ with $N$ simply connected nilpotent Lie group and $K$ compact; $N$ is the nilpotent radical of $H/E$ and thus is characteristic. Denoting by $p$ the projection $H\to H/E$, Lemma~\ref{lem:quot-Lie} implies that $p(C)=(H/E)^\circ$, so $p^{-1}((H/E)^\circ)=E\times C$. Thus $E\times C\to (H/E)^\circ$ is isomorphic to the projection $E\times C\to C$. In particular, $E\times N=p^{-1}(N)$, and $N$ can be identified with the subgroup $p^{-1}(N)^\circ$, so $p^{-1}(N)$ is generated by $p^{-1}(N)^\circ$ and $\Ker(p)\subset p^{-1}(K)$. Thus $H$ is generated by the simply connected nilpotent Lie group $N_1=p^{-1}(N)^\circ$ (which is characteristic) and the locally elliptic group $E_1=p^{-1}(K)$, and the claim is proved.
\end{proof}

\begin{proof}[Proof of Proposition~\ref{someall}]
Let $\pi : G \to \Lambda$ be the projection homomorphism. By Lemma~\ref{inner}, every element in $\pi^{-1}(\pi(\alpha))$ is compacting. In particular there is no loss of generality in assuming that $\alpha \in \Lambda$. Since an element is compacting if and only if some/every positive power is compacting (see Lemma~\ref{lem:cbp}), we are done if $\Lambda=\ZZ$. In case $\Lambda=\RR$, we let $\varphi : \RR \to \Lambda : t \mapsto \varphi_t$ be an isomorphism with $\varphi_1 = \alpha$ and $C \subseteq H $ be a compact vacuum set for $\alpha$ (see Lemma~\ref{lem:cbp}). Then $K = \bigcup_{0 \leq s \leq 1} \varphi_s(C)$ is compact and contains $\varphi_t^n(h)$ for all $h \in H$ and all sufficiently large $n \geq 0$. Thus every element of $\Lambda_+$ is compacting, whence the claim by Lemma~\ref{inner}.
\end{proof}

\subsection{The closure of a contraction subgroup}

Recall that to any automorphism $\alpha $ of a locally compact group $G$, one can associate the \textbf{contraction subgroup} $U_\alpha = \{g \in G \; | \;  \lim_n \alpha^n g  =1\}$. It is important to point out that this contraction subgroup need not be closed in general: an excellent illustration of this fact is provided by the group $G =\Aut(T)$ of all automorphisms of a regular locally finite tree. It is easy to see that the contraction subgroup $U_\alpha$ associated to any hyperbolic element $\alpha \in G$ is never closed. It is therefore natural to study what the closure of the contraction subgroup $U_\alpha$ can be. The following observation, which will be used in the proof of Theorem~\ref{thm:application} therefore provides additional motivation to consider \emph{compacting} automorphisms. 

\begin{prop}\label{prop:ClosureContraction}
Let $G$ be a  locally compact group and $\alpha \in \Aut(G)$. 

Then the restriction of $\alpha$ to the closure $\overline{U_\alpha}$ acts as a compacting automorphism. 
\end{prop}

\begin{proof}
Clearly the group $\overline{U_\alpha}$ is invariant under $\alpha$. There is thus no loss of generality in assuming that $G = \overline{U_\alpha}$. 

We next observe that if $H \leq G$ is any $\alpha$-invariant closed normal subgroup, then $G/H$ is a locally compact group such that the contraction subgroup $U_\alpha^{G/H}$ associated to the automorphism of $G/H$ induced by $\alpha$ is dense. In view of this observation, there is no loss of generality in assuming that the maximal compact normal subgroup of the neutral component $G^\circ$ is trivial. In particular $G^\circ$ is a Lie group. 

\medskip
We next apply this observation to the identity component $H = G^\circ$. The quotient $G/G^\circ$ is totally disconnected. Given a compact open subgroup  $V \leq G/G^\circ$, we set   $V_- = \bigcap_{n \geq 0} \alpha^{-n}(V)$ and $V_{--} = \bigcup_{n \geq 0} \alpha^{-n}( V_-)$. Clearly $V_-$ is compact. According to Theorem~1 from~\cite{Willis94}, we can find a compact open subgroup $V$ such that $V_{--}$ is closed in $G/G^\circ$. 

Let $u \in U_\alpha^{G/G^\circ}$. Since $V$ is open, there is some $N$ such that $\alpha^n(u) \in V$ for all $n \geq N$. Thus $\alpha^N(u) \in V_-$ and, hence, we have $u \in \alpha^{-N}(V_-) \leq V_{--}$. This proves that $ U_\alpha^{G/G^\circ} $ is contained in $V_{--}$. Since the latter is closed, we deduce that 
$$
G/G^{\circ} = \overline{U_\alpha^{G/G^\circ} } \leq V_{--}
$$
and, hence, that $\alpha$ acts on $G/G^\circ$ as a compacting automorphism: indeed, for all $g \in  G/G^{\circ} $ we have $\alpha^n(g) \in V_-$ for all sufficiently large $n$.

\medskip
We have shown that $G/G^\circ$ admits a compacting automorphism. By Proposition~\ref{prop:stcom}, this implies that $G/G^\circ$ is locally elliptic. Since the neutral component $G^\circ$ has no nontrivial compact normal subgroup, we can thus invoke Theorem~A.5 from~\cite{CoTe}, which ensures that $G$ has a characteristic open subgroup splitting as a direct product $J \cong G^\circ \cdot D$, where $D \leq G$ is a closed, totally disconnected, locally elliptic and characteristic in $G$. Since the quotient $G/J$ is discrete and contains a dense contraction subgroup by the preliminary observation above, it must be trivial. Thus $G \cong G^\circ \times D$. Since $G^\circ \cong G/D$ is a Lie group with no nontrivial compact normal subgroup, the contraction subgroup $U_\alpha^{G/D}$ is closed in $G/D$ by Proposition~\ref{prop:cpli}. Therefore $\alpha $ acts on $G^\circ$ as a contracting automorphism. Thus we have $G^\circ \leq U_\alpha $ and the desired result follows since $G = G^\circ \times D$ and since we have already established that $\alpha$ acts on the totally disconnected group $D$ as a compacting automorphism. 
\end{proof}

\subsection{Decompositions of compactions}

We can wonder whether a converse to Proposition~\ref{prop:ClosureContraction} holds. Here is a precise formulation of this question:
Given a locally compact group $H$ and a compaction $\alpha$ of $H$, do we always have $H=KU_\alpha$, where $K$ is the limit group of the compaction and $U_\alpha$ the contraction subgroup of $\alpha$? By Baumgartner-Willis \cite[Cor.~3.17]{BaumgartnerWillis} {(see \cite{Jaworski} for the general case of possibly non-metrizable groups)}, this is true when $H$ is totally disconnected and this easily extends (in view of Propositions~\ref{prop:stcom} and~\ref{prop:cpli}) to the case when $H^\circ$ has no nontrivial compact normal subgroup. A positive answer would simplify the statement in Theorem~\ref{thm:application}(\ref{i_Hilbert}) below, avoiding the need to mod out by a compact normal subgroup.

\section{{Amenable hyperbolic groups and millefeuille spaces}}\label{prooff}

The purpose of this section is to establish a    sharpening of Theorem~\ref{thm:Amen:intro}, namely Theorem~\ref{thm:main} below. Theorem~\ref{thmintro:cat-1} and Corollary~\ref{cor:woesskai} from the introduction will then follow easily. This first requires a discussion of millefeuille spaces.

\subsection{Millefeuille spaces}\label{sec:millefeuille}

Fix $-\kappa\le 0$. Let $X$ be a complete CAT($-\kappa$) metric space and $b:X\to\RR$ a surjective, 1-Lipschitz convex function. 


For example if we set $b(x)=\lim_{n\to +\infty}(d(x,x_n)-d(x_0,x_n))$ for some sequence $x_n$ tending to infinity along a geodesic ray,
then $b$ satisfies these conditions and is a Busemann function (see Proposition~II.8.19 in~\cite{bridson1999metric}).

For $k\ge 1$, we define a new CAT($-\kappa$)-space $X[k]$ as follows. 
Let $T$ be the $(k+1)$-regular tree (identified with its 1-skeleton) with a surjective Busemann function $b'$ (taking integral values on vertices). As a topological space
$$X[k]=\{(x,y)\in X\times T:\;b(x)=b'(y)\}.$$
Note that $X[1]=X$.
Note that in case $X$ is a $d$-dimensional Riemannian manifold the map $b: X\to\RR$ is a trivial bundle with fibre $\RR^{d-1}$ and thus $X[k]$ is homeomorphic to $\RR^{d-1}\times T$ and in particular is contractible.

Locally, $X[k]$ is obtained from $X$ by gluing finitely many CAT($-\kappa$) spaces along closed convex subsets, and thus (see \cite[II.11]{bridson1999metric}) $X[k]$ is canonically endowed with a locally CAT($-\kappa$) metric; it is called the {\bf millefeuille space} of degree $k$ associated to $X$. This metric defines the same uniform structure and bornology as the metric induced by inclusion, and thus in particular $X[k]$ is a complete metric space and if $X$ is proper then so is $X[k]$.

Now let $G$ be a locally compact group with a homomorphism $p$ onto $\ZZ$, and isometric actions on $X$ and $T$ satisfying $b(gx)=b(x)+p(g)$ for all $(g,x)\in G\times X$ and $b'(gx)=b'(x)+p(g)$ for all $(g,x)\in G\times T$ (we say that $b$ and $b'$ are {\bf equivariant} with respect to $G$). Then the product action of $G$ on $X\times T$ preserves $X[k]$ and preserves its metric.

The following result is one of the essential steps in the proof of Theorem~\ref{thm:main}, and provides a precise formulation of the geometric consequences that can be derived from Theorem~\ref{thm:main}(\ref{i_str}). We isolate its statement to emphasize the specific role of the millefeuille space (the case of a semidirect product $H\rtimes\RR$ was already considered in Corollary~\ref{Heintzeco}).

\begin{thm}\label{thm:actmilf}
Let $G=H\rtimes_\alpha\ZZ$ be a locally compact group, where $\ZZ$ acts by a compaction $\alpha$ of $H$. Then
\begin{enumerate}[(a)]
\item\label{glf} $G/\LF(G)$ acts properly cocompactly by isometries on a negatively curved manifold $X$, with an equivariant Busemann function $b$, and the projection $X\stackrel{b}\to\RR\to\RR/\ZZ$ identifies $G\backslash X$ with the circle $\RR/\ZZ$;
\item\label{ggci} $G/G^\circ$ acts properly, vertex and edge-transitively on a $(k+1)$-regular tree for some $k$, with an equivariant Busemann function $b'$;
\item\label{concerto_pour_violon} the corresponding product action of $G$ restricts to a proper cocompact action on the negatively curved millefeuille space $X[k]$, the projection $X[k]\stackrel{\beta}\to\RR\to\RR/\ZZ$ identifying the orbit space $G\backslash X[k]$ with the circle $\RR/\ZZ$.  
\end{enumerate}
\end{thm}
\begin{proof}
By Proposition~\ref{prop:stcom}, $H/\LF(G)$ is connected-by-compact, and (\ref{glf}) follows from Corollary~\ref{Heintzeco}. 

The group $G/G^\circ$ is totally disconnected and (\ref{ggci}) follows from Lemma~\ref{lem:TD_compactable} (which also gives the explicit value of $k$).

It follows that for each $r$, the product action of $H$ preserves and is transitive on $$\{(x,y)\in X\times T|\;b(x)=r \text{ and } b'(y)=r\};$$
since the generator of $\ZZ$ adds $(1,1)$ to $(b(x),b(y))$, it follows that for every $r\in\RR$, the action of $G$ is transitive on $\{(x,y)\in X\times T|b(x)=b'(y)\in\ZZ+r\}\subset X[k]$ and in particular is cocompact on $X[k]$. 

By Proposition~\ref{prop:stcom}, modulo some compact normal subgroup of $G^\circ$, the subgroup generated by the kernels $\LF(G)$ and $G^\circ$ is a topological direct product; it follows that the product action is proper.
\end{proof}

\begin{remark}
Given two hyperbolic spaces with Busemann functions $(X,b)$ $(Y,b')$, a notion of {\bf horocyclic product} 
$$\{(x,y)\in X\times Y:\;b(x)+b'(y)=0\}$$
was introduced by Woess and studied by various authors. In the case of two trees, it is known as {\bf Diestel--Leader graph}. Despite an obvious analogy, millefeuille spaces are not horocyclic products. Actually, horocyclic products are never hyperbolic, except in a few degenerate uninteresting cases. 
\end{remark}

\subsection{A comprehensive description of amenable hyperbolic groups}\label{proofaa}

The following statement, where all actions and homomorphisms are implicitly assumed to be continuous, is a more comprehensive version of Theorem~\ref{thm:Amen:intro}: indeed, the latter is covered by the equivalence (\ref{i_hyp})$\Leftrightarrow$(\ref{i_str}).

\begin{thm}\label{thm:main}
Let $G$ be a locally compact group. 
Then the following assertions are equivalent.
\begin{enumerate}
\item\label{i_qp} $G$ is focal hyperbolic.
\item\label{i_unimod} $G$ is amenable, hyperbolic and non-unimodular. 

\item\label{i_hyp} $G$ is amenable and non-elementary hyperbolic.

\item\label{i_com} $G$ is $\sigma$-compact and there exists a homomorphism $\beta:G\to\RR$ with closed image, such that $\Ker(\beta)$ is noncompact and for some $x$ with $\beta(x)\neq 0$, the action by conjugation of $x$ on $\Ker(\beta)$ is compacting.

\item\label{i_str} One of the two following holds:
\begin{itemize}
\item $G=N\rtimes(K\times\RR)$ where $N$ is a nontrivial simply connected nilpotent Lie group on which $\RR$ acts by contractions, and $K$ is compact subgroup. 
\item $G=H\rtimes\ZZ$, where $H$ is closed, noncompact, and the action of $\ZZ$ on $H$ is compacting.
\end{itemize}

\item\label{i_fib} $G$ has actions by isometries on a homogeneous negatively curved manifold $X$ and on a regular tree $T$ with $G$-equivariant surjective Busemann functions $b$ and $b'$ ($X$ and $T$ not being both reduced to a line), so that the action on the product $X\times T$ is proper and preserves cocompactly the fibre product 
$$\{(x,y)\in X\times T \; | \;b(x)=b'(y)\}.$$

\item\label{i_cat} $G$ acts properly and cocompactly by isometries on a proper  geodesically complete  \cathyp~space $X \not \cong \RR$ and fixes a point in the visual boundary $\bd X$.
\end{enumerate}
\end{thm}

The final subsidiary fact needed for the proof is the following.

\begin{lem}\label{lem:directr}
Let $G$ be a locally compact group and $W$ a compact normal subgroup so that $G/W\simeq\RR$. Then $G$ can be written as a direct product $W\times\RR$.
\end{lem}
\begin{proof}
Since the outer automorphism group of $W$ is totally disconnected, the $G$-action by conjugation on $W$ is by inner automorphisms, so $WZ=G$ where $Z$ is the centralizer of $W$. If $p : G \to \RR$ is the projection modulo $W$, by an easy case of Lemma~\ref{lem:quot-Lie} we have $p(Z^\circ)=\RR$. By Theorem~4.15.1 in  \cite{montgomery1955topological}, the group  $Z^\circ$ contains a one-parameter subgroup $P$ such that $p(P)=\RR$. It follows that $G=W\times P$.
\end{proof}

\begin{proof}[Proof of Theorem~\ref{thm:main}]
We shall prove that (\ref{i_qp})$\Leftrightarrow$(\ref{i_hyp})$\Rightarrow$(\ref{i_com})$\Rightarrow$(\ref{i_str})$\Rightarrow$(\ref{i_fib})$\Rightarrow$(\ref{i_cat})$\Rightarrow$(\ref{i_hyp}), and (\ref{i_unimod})$\Leftrightarrow$(\ref{i_hyp}). 
A direct, independent and conceptually different approach for the implication (\ref{i_com})$\Rightarrow$(\ref{i_hyp}) is provided, in a much more general setting, by Theorem~\ref{thm:Contracting}.

\medskip\noindent
(\ref{i_qp})$\Leftrightarrow$(\ref{i_hyp}) Since $G$ is focal, it is non-elementary by definition and amenable by Lemma~\ref{lem:AmyStab}. Conversely if $G$ is amenable and non-elementary hyperbolic, it is either focal or general type by Proposition~\ref{prop:types}, but cannot be of general type since otherwise it could contain a discrete nonabelian free subgroup by Lemma~\ref{lem:schottky}, contradicting amenability.

\medskip \noindent 
(\ref{i_unimod})$\Rightarrow$(\ref{i_hyp}) Remark that a horocyclic action is never cocompact (see Proposition~\ref{prop:undistorted}). Therefore an elementary hyperbolic locally compact group is either bounded or lineal, and it follows that it must contain a cyclic group as a uniform lattice (see Proposition~\ref{pro:debout} if necessary). In particular, it must be unimodular, and the desired implication follows.

\medskip \noindent 
(\ref{i_hyp})$\Rightarrow$(\ref{i_com}) follows from Proposition~\ref{prop:coctamen} and the fact that any hyperbolic locally compact group admits a continuous, proper cocompact action by isometries on a proper hyperbolic geodesic metric space (See Corollary~\ref{cor:eqprop}).

\medskip \noindent 
(\ref{i_com})$\Rightarrow$(\ref{i_str}) If $\beta$ has cyclic image, the desired statement is trivial. We assume henceforth that  $\beta$ has non-discrete image. Thus $\beta$ is surjective. Since $G$ is $\sigma$-compact, the homomorphism $\beta$ is open and is thus a quotient map. Thus it follows from Lemma~\ref{lem:quot-Lie} that $\beta(G^\circ) = \RR$. By Yamabe's theorem (see Theorem~\ref{thm:Yamabe})  $G^\circ$ is a projective limit of Lie groups. Therefore, Theorem~4.15.1 from ~\cite{montgomery1955topological} implies that  there exists some one-parameter subgroup $P < G^\circ$ such that  $\beta(P) = \RR$. By Proposition~\ref{someall}, the action of $P$ on $\Ker(\beta)$ is compacting. By Lemma~\ref{lem:cbp}(\ref{continuous}), it follows that $G$ is connected-by-compact. By Proposition~\ref{hazsie}, we can write $G=H\rtimes (K\rtimes P)$. By Lemma~\ref{lem:directr} below, $K\rtimes P$ can be rewritten as a direct product $K\times\RR$.

\medskip \noindent 
(\ref{i_str})$\Rightarrow$(\ref{i_fib}) In the first case when $G$ maps onto $\RR$, Corollary~\ref{Heintzeco} directly applies.

In the second case, we have $G=H\rtimes\ZZ$ and Theorem~\ref{thm:actmilf} applies.

\medskip \noindent 
(\ref{i_fib})$\Rightarrow$(\ref{i_cat}) This follows from the remarks preceding Theorem~\ref{thm:actmilf}.

\medskip \noindent 
(\ref{i_cat})$\Rightarrow$(\ref{i_hyp}) Recall from Lemma~\ref{lem:AmyStab} that if $X$ is a proper hyperbolic metric space of a cocompact isometry group (or, more generally, of bounded geometry), then the stabilizer $\Isom(X)_\xi$ of every point $\xi \in \bd X$ is amenable. Thus (\ref{i_hyp})  follows from (\ref{i_cat}) since any \cathyp space is hyperbolic. 

\medskip \noindent 
(\ref{i_hyp})$\Rightarrow$(\ref{i_unimod}) The only thing to check is that $G$ is non-unimodular. We know that (\ref{i_hyp}) implies (\ref{i_str}). Thus it suffices to observe that a group $G$ satisfying  (\ref{i_str}) cannot be unimodular. This follows from Proposition~\ref{compamenable}.
\end{proof}

Clearly, Theorem~\ref{thmintro:cat-1} is immediate from Theorem~\ref{thm:main}(\ref{i_cat}).

\begin{proof}[Proof of Corollary~\ref{cor:woesskai}]
Let $\Gamma$ be a group acting vertex-transitively on a hyperbolic locally finite graph and fixing a point at infinity,  and let $G$ be its closure in the full automorphism group of the graph. Then $G$ is hyperbolic totally disconnected and fixes a point on its boundary, so, if non-elementary (the elementary case being trivial), it satisfies the properties of Theorem~\ref{thm:main}; in particular, it can be written as $N\rtimes\ZZ$ with a compacting action of $\ZZ$ on the totally disconnected group $N$. It follows by Lemma~\ref{lem:TD_compactable} that $G$ is quasi-isometric to a regular tree.
\end{proof}


\section{Characterizing standard rank one groups}\label{sec:rankone}

Standard rank one groups were defined after Theorem~\ref{thm:Clight} in the introduction. The goal of this section is to prove the following statement. 

\begin{thm}\label{thm:application}
Let $G$ be a locally compact group. 
Then the following assertions are equivalent. 

\begin{enumerate}[(i)]
\item\label{i_HypUnim} 
$G$ is hyperbolic of general type and contains a cocompact amenable closed subgroup.  

\item\label{i_HypNonAmen} 
$G$ is non-elementary hyperbolic and the action of $G$ on its visual boundary is transitive.

\item\label{i_HypNonAmen2} 
$G$ is non-elementary hyperbolic and the action of $G$ on its visual boundary is 2-transitive.

\item\label{i_Hilbert}
$G$ is unimodular and contains a compact normal subgroup $W$ such that $G/W$ contains an element $\alpha$ so that, defining $U_\alpha := \{g \in G/W \; | \; \lim_{n \to \infty} \alpha^n g \alpha^{-n} =1\}$ as the \emph{contraction subgroup} associated with $\alpha$, the subgroup $U_\alpha$ has noncompact closure and the closed subgroup $\overline{\la \alpha \ra U_\alpha}$ is cocompact in $G/W$.

\item\label{i_Classif} 
$G$ has a maximal compact normal subgroup $W$, and $G/W$ is a standard rank one group. 
\end{enumerate}
\end{thm}

Notice that the condition (\ref{i_Hilbert}) in Theorem~\ref{thm:application} is purely in the language of the category of locally compact groups: no geometric or analytic property is involved. This feature of Theorem~\ref{thm:application} is in fact shared by Theorem~\ref{thm:Amen:intro}: the geometric condition of negative curvature is deduced from a condition of algebraic/topological nature through the concept of contraction or {compaction}. 

\medskip
Let us point out that Theorem~\ref{thm:Clight} from the introduction follows immediately:

\begin{proof}[Proof of Theorem~\ref{thm:Clight}]
All elementary hyperbolic groups are amenable. Moreover, by Lemma~\ref{lem:AmyStab}, a focal hyperbolic group also is amenable. Thus a non-amenable hyperbolic group is non-elementary of general type. The implication (\ref{i_HypUnim})$\Rightarrow$(\ref{i_Classif}) from Theorem~\ref{thm:application} therefore yields the desired conclusion.
\end{proof}

\subsection{Proof of Theorem~\ref{thm:application}}\label{proofb}

The proof requires the following classical lemma {(see the corollary to Theorem~8 in \cite{Arens} or Ch.~VII, App.~1, Lemme~2 in \cite{Bourbaki}).

\begin{lem}\label{lem:orbihomeo}Let $G$ be a locally compact, $\sigma$-compact group $G$ and $X$ a locally compact topological space on which $G$ acts continuously and transitively. Then for every $x\in X$, the orbital map $G/G_x\to X$ is a homeomorphism.\qed
\end{lem}

\begin{proof}[Proof of Theorem~\ref{thm:application}]
We first show that (\ref{i_HypUnim})  $\Leftrightarrow$ (\ref{i_HypNonAmen}) $\Leftrightarrow$ (\ref{i_HypNonAmen2}), because this is based only on general arguments of hyperbolic geometry from Section~\ref{sec:AmenableActions}. Throughout the proof, we let $X$ be a proper geodesic hyperbolic space on which $G$ acts continuously, properly and cocompact by isometries, see Corollary~\ref{cor:eqprop}.

\medskip \noindent
(\ref{i_HypUnim}) $\Rightarrow$ (\ref{i_HypNonAmen2}) Let us first check that the action on the visual boundary is transitive. Indeed, let $H$ be a cocompact amenable subgroup. By Proposition~\ref{prop:coctamen}(\ref{utrans}), the action of $H$ on the visual boundary has exactly two orbits, namely one singleton $\{\omega\}$ and its complement. Since $G$ is of general type, it has no fixed point and therefore it has only one orbit.
Since by Proposition~\ref{prop:coctamen} the stabilizer $H$ is transitive on $\bd X-\{\omega\}$, we deduce that the action on the visual boundary is 2-transitive.

\medskip \noindent
(\ref{i_HypNonAmen}) $\Rightarrow$ (\ref{i_HypUnim}). Since $\bd X$ is compact and the action on $\bd X$ is continuous, and $G$ is $\sigma$-compact, the stabilizer $H$ of a point is cocompact by Lemma~\ref{lem:orbihomeo}. By Lemma~\ref{lem:AmyStab}, $H$ is amenable.

\medskip \noindent
(\ref{i_HypNonAmen2}) $\Rightarrow$ (\ref{i_HypNonAmen}) is trivial.

\medskip

Now (\ref{i_HypUnim})  $\Leftrightarrow$ (\ref{i_HypNonAmen}) $\Leftrightarrow$ (\ref{i_HypNonAmen2}) is granted and we prove the equivalence of these properties with the last two ones. That (\ref{i_Classif}) implies (\ref{i_HypNonAmen}) is immediate. To establish the equivalence of Properties (\ref{i_HypUnim}), (\ref{i_HypNonAmen}), (\ref{i_HypNonAmen2}) with (\ref{i_Classif}) we now tackle the following implication.

\medskip \noindent
(\ref{i_HypNonAmen2}) $\Rightarrow$ (\ref{i_Classif})
Clearly the $G$-action on $X$ cannot be focal, so $G$ is of general type. By Lemma~\ref{lem:CptNorm}, the group $G$ has a maximal compact normal subgroup $W$. Upon replacing $G$ by $G/W$, there is no loss of generality in  assuming that $W$ is trivial. 

By Proposition~\ref{pro:gt}, $G$ is either virtually a simple adjoint Lie group of rank one (of the specified type), or is totally disconnected. In the former case (\ref{i_Classif}) follows and are done. We thus henceforth assume that $G$ is totally disconnected.

Let $\omega \in \bd X$ be a point at infinity and set $ H = G_\omega$. By assumption the orbit map induces a continuous bijection $\pi : G/H \to \bd X$, which is a homeomorphism by Lemma~\ref{lem:orbihomeo}.


Since $G$ is Gromov-hyperbolic, it is compactly generated. Recall that any compactly generated totally disconnected locally compact group acts continuously, properly and vertex-transitively by automorphisms on some locally finite connected graph $\mathfrak g$, namely its Cayley--Abels graph (see  the discussion following Proposition~\ref{prop:Cayley}). Since $\mathfrak g$ is quasi-isometric to $G$, hence to $X$, we may assume without loss of generality that $\mathfrak g = X$ or, equivalently, that $X$ is a graph. Since $G/H \cong \bd X$, it follows that $\bd X$ is totally disconnected. Thus $X$ cannot be one-ended, and must therefore have  infinitely many ends.  Since the set of ends of $X$ is a quotient of $\bd X$, we deduce that $G$ acts transitively on the set of ends of $X$. By a result independently due to M\"oller \cite{moller1992ends} and Nevo \cite{Nevo}, there exists an equivariant quasi-isometry from $X$ to a locally finite tree $T$ on which $G$ acts  properly and cocompactly. We can further suppose that this action is minimal and without inversions. We can also suppose that there is no erasable vertex in $T$ (a vertex is \textbf{erasable} if it has degree two and its stabilizer fixes both edges emanating from it). Observe that $G$ is the fundamental group of a certain finite graph of groups with universal covering $T$, so that in the graph of groups, all vertex and edge groups are profinite groups and inclusions are open. 

It is easy to check that for every vertex $v$ of degree at least three, the action of $G_v$ on the set $E(v)$ of neighbouring edges of $v$ is 2-transitive (see~\cite{burger2000groups}). This moreover holds also for vertices of degree two, since no such vertex is erasable. It follows that the graph of groups is actually an edge. This edge cannot be a loop, since otherwise the $G$-action on $T$ would be focal, which is excluded. 

\medskip
 It remains to establish the equivalence between (\ref{i_Hilbert}) and the other properties. This follows from the following two implications.

\medskip \noindent [(\ref{i_HypUnim}) and (\ref{i_Classif})]$\Rightarrow$(\ref{i_Hilbert})
We start from (\ref{i_HypUnim}). By Lemma~\ref{lem:CptNorm}, $G$ has a maximal compact normal subgroup $W$. 
So by assumption, $G/W$ has a closed amenable cocompact subgroup $G_1$. Being quasi-isometric to $G$, the group $G_1$ is non-elementary hyperbolic and therefore is, by Theorem~\ref{thm:main}(\ref{i_str}),
of the form $H\rtimes\Lambda$, where $\Lambda\in\{\ZZ,\RR\}$ acts by compacting $H$. 
Using Propositions~\ref{prop:stcom} and~\ref{prop:cpli}, $G_1$ admits a closed cocompact subgroup $G_2$ of the form $(N\times E)\rtimes_\alpha\ZZ$, where $N$ is a simply connected nilpotent Lie group and $E$ is totally disconnected, the action of $\ZZ$ preserving the direct decomposition $N\times E$, contracting $N$ and compacting $E$. Set $$U_\alpha(N\times E)=\{g\in N\times E:\lim_{n\to+\infty}\alpha^n(g)=1\}=N\times U_\alpha(E).$$
Since $E$ is totally disconnected, Corollary~3.17 from Baumgartner-Willis \cite{BaumgartnerWillis} implies that the closure of $U_\alpha(E)$ is cocompact in $E$. So the group $G_3=\overline{U_\alpha(N\times E)}\rtimes_\alpha\ZZ$ is a closed, cocompact subgroup of $G/W$ of the required form.

We finally use (\ref{i_Classif}), which implies that the abelianization of $G$ is compact, to deduce that $G$ is unimodular.

\medskip \noindent 
(\ref{i_Hilbert})$\Rightarrow$(\ref{i_HypUnim})
Let $G_1$ be the subgroup $\overline{\langle\alpha\rangle U_\alpha}$. We are going to show that $G_1$ is hyperbolic, amenable, and non-unimodular. Taking this for granted, this implies that $G_1$ is non-elementary hyperbolic by Theorem~\ref{thm:main}, so that $G$ is also non-elementary hyperbolic. But $G$ cannot be amenable since otherwise by Theorem~\ref{thm:main}, $G$ would be non-unimodular, so $G$ is of general type and admits the amenable group $G_1$ as a closed amenable cocompact subgroup.

Let us now check that $G_1$ is indeed hyperbolic, amenable, and non-unimodular.
Note that $\overline{U_\alpha}$ is normal in $G_1$. By Proposition~\ref{prop:ClosureContraction}, the action of $\alpha$ on $\overline{U_\alpha}$ is compacting; moreover $\overline{U_\alpha}$ is noncompact by assumption. In particular, the Haar multiplication of $\alpha$ on $\overline{U_\alpha}$ is less than one, while its Haar multiplication on the abelian quotient $G_1/\overline{U_\alpha}$ is obviously trivial. Thus $\Delta_{G_1}(\alpha)\neq 1$ and $G_1$ is non-unimodular. By Theorem~\ref{thm:main}(\ref{i_com}) (with $\beta=\log\Delta_{G_1}$), we deduce that $G_1$ is also hyperbolic. Finally, $G_1$ is amenable by Proposition~\ref{compamenable}.
\end{proof}

\subsection{Relatively hyperbolic non-uniform lattices}

Theorem~\ref{thm:RelHyp:cocpt} from the introduction follows readily from the combination of Theorem~\ref{thm:application} and the following. 

\begin{prop}\label{prop:RelHyp:cocpt}
Let $X$ be a proper hyperbolic geodesic metric space with cocompact isometry group, and let $\Gamma \leq \Isom(X)$ be  a closed cusp-uniform subgroup. 

If the  $\Gamma$-action is not cocompact, then $\Isom(X)$ is doubly transitive on $\bd X$. 
\end{prop}

\begin{proof}
If $\Isom(X)$ stabilizes a point or a pair of points in $\bd X$, then it is amenable by Lemma~\ref{lem:AmyStab}, and so is $\Gamma$. It follows that $\Gamma$ acts cocompactly by Proposition~\ref{prop:RelHyp:bis}, and we are done.

We may thus assume that $\Isom(X)$ is of general type, and that the $\Gamma$-action is not cocompact. Thus there is some $\xi \in \bd X$ which is bounded parabolic and not conical with respect to $\Gamma$. 

Our goal is to show that there is some $\eta \in \bd X$ such that the stabilizer $P = \Isom(X)_\eta$ is transitive on on $\bd X \setminus \{\eta\}$. Since $\Isom(X)$ is of general type, it follows that $\Isom(X)$ is doubly transitive on $\bd X$, as desired. 

Since $\xi$ is not conical with respect to $\gamma$, we have $\beta_\xi(\Gamma_\xi) = 0$. This means that every horosphere around $\xi$, defined with respect to some choice of horofunction, say $h$, is preserved by $\Gamma_\xi$ up to some fixed constant. Let now $\rho : \RR_+ \to X$ be a geodesic ray pointing to $\xi$. Since $\Isom(X)$ is cocompact, we can find sequences $(g_n)$ in $G$ and $(t_n)$ in $\RR_+$ such that $\lim_n t_n = \infty$, $\lim g_n(\rho(t_n)) = x_0$ for some $x_0 \in X$, and $g_n.\rho$ converges uniformly on compact sets to some geodesic line $\ell$ passing through $x_0$. In particular, one of the two endpoints of $\ell$, say $\eta$, coincides with $\lim_n g_n.\xi$. Let $P = \Isom(X)_\eta$.

\medskip
We claim that there is some constant $R$ such that $P.\mathcal N_R(\ell) = X$, where $\mathcal N_R(\ell)$ denotes the $R$-neighbourhood of $\ell$. Since $\eta$ is an endpoint of $\ell$, this immediately implies that $P$ acts transitively on $\bd X \setminus \{\eta\}$, so the claim implies the theorem.

\medskip
Indeed, let $x \in X$ be arbitrary. Since $g_n\inv(x) $ converges to $\xi$ within some bounded neighbourhood of the ray $\rho$, we have $h(g_n\inv.x) = -\infty$, where $h$ is the horofunction which was chosen above. Now we observe that the assumption that $\xi$ is bounded parabolic implies that there is some constant $R$ such that for all $z \in X$, if $h(z) <0$ then there is some $\gamma \in \Gamma$ such that $\gamma.z$ is a distance at most $R$ from $\rho$. In particular, for all $n$ large enough, we can find $\gamma_n \in \Gamma_\xi$ such that $\gamma_n g_n\inv(x) \in \mathcal N_R(\rho)$.  Thus 
\begin{equation}\label{eq:conj}
g_n \gamma_n g_n\inv(x) \in \mathcal N_R(g_n.\rho). 
\end{equation}
Combining the fact that $g_n\inv(x)$ remains at bounded distance from $\rho$ and 
$\beta_\xi(\gamma_n) = 0$, we deduce that  $d(\gamma_n  g_n\inv(x),g_n\inv(x))$ is bounded.  Hence the sequence $g_n \gamma_n g_n\inv$ is bounded in $\Isom(X)$ and we can assume, upon passing to a subsequence, that it converges to some $h \in \Isom(X)$. Since $\lim_n g_n\xi  = \eta$, we have $h \in P$. Passing to the limit as $n \to \infty$ in (\ref{eq:conj}), we obtain $h(x) \in \mathcal N_R(\ell)$, which proves the claim. 
\end{proof}

%
%
%
%

\bibliographystyle{amsalpha}
\bibliography{amenhyp}
\end{document}